\documentclass{arxiv}

\usepackage{amssymb}
\usepackage{algorithm}
\usepackage{float}
\usepackage{needspace}
\usepackage{algorithmic}
\usepackage{subcaption}
\usepackage{xcolor}
\usepackage{enumerate}

\usepackage{amsfonts}
\usepackage{graphicx}
\usepackage{epstopdf}
\usepackage{algorithm}
\usepackage{algorithmic}
\ifpdf
  \DeclareGraphicsExtensions{.eps,.pdf,.png,.jpg}
\else
  \DeclareGraphicsExtensions{.eps}
\fi

\AtBeginDocument{}

\newsiamremark{remark}{Remark}
\newsiamremark{hypothesis}{Hypothesis}
\crefname{hypothesis}{Hypothesis}{Hypotheses}
\newsiamthm{claim}{Claim}
\newsiamremark{fact}{Fact}
\crefname{fact}{Fact}{Facts}

\headers{Adaptive projected-gradient SAA for stochastic MOO}{Y. Li, L. Wang, and X. Chen}

\title{An Adaptive Projected-Gradient Algorithm for Sample-Average Approximations of Stochastic Multi-Objective Optimization}

\author{%
  Yiyang Li\thanks{yiyang-amber.li@connect.polyu.hk}
  \qquad
  Lei Wang\thanks{wlkings@lsec.cc.ac.cn}
  \qquad
  Xiaojun Chen\thanks{maxjchen@polyu.edu.hk}
}

\affiliation{%
  Department of Applied Mathematics,
  The Hong Kong Polytechnic University,
  Hong Kong, China
}

\date{}

\usepackage{amsopn}

\ifpdf
\hypersetup{
  pdftitle={An Adaptive Projected-Gradient Algorithm for Sample-Average Approximations of Stochastic Multi-Objective Optimization},
  pdfauthor={Yiyang Li}
}
\fi
\newtheorem{assumption}{Assumption}

\newcommand{\st}{\mathrm{s.\,t.}}

\begin{document}

\maketitle

\begin{abstract}
We consider stochastic multi-objective optimization over a nonempty closed convex set, where every objective is an expectation and only sample-gradient information is available.
We develop a line-search-free and function-value-free adaptive projected-gradient algorithm for the sample-average approximation (SAA) problem.
Each iteration computes a feasible regularized multi-gradient step and updates the regularization parameter from the projected step length. 
A normal-cone-based certificate yields descent estimates and an explicit complexity bound for the Pareto-stationarity residual of the SAA problem. 
The consistency of SAA gradients then transfers vanishing SAA residuals to Pareto stationarity for the population problem, while an additional concentration argument gives a finite-sample residual bound on compact sets.
Experiments on synthetic problems, classification, portfolio selection, multi-task learning, and robot control illustrate the practical performance of our algorithm. 
\end{abstract}


%
%
%
%
%
%


\section{Introduction}

\label{sec:introduction}

Many learning and decision-making tasks involve competing objectives, including predictive accuracy versus model complexity, fit versus robustness to distributional shifts, and expected return versus financial risk.
A multi-objective formulation \citep{ehrgott2026fifty} represents these trade-offs directly and seeks solutions that are balanced in the sense of Pareto optimality. 
In practice, the objectives are usually stochastic since they depend on an underlying data-generating distribution that can be accessed only through simulation or repeated observations. 

Let $\mathcal X\subseteq\mathbb R^n$ be a nonempty, closed, and convex set. 
We consider the stochastic multi-objective optimization (SMOO) problem
\begin{equation}
	\label{pb:multi}
	\min_{\mathbf z\in\mathcal X} \; F(\mathbf z):=\big(F_1(\mathbf z),\ldots,F_m(\mathbf z)\big)^\top.
\end{equation}
For each $i\in[m]:=\{1,\ldots,m\}$, the function $F_i$ is defined as
\begin{equation*}
	F_i(\mathbf z):=\mathbb E[f_i(\mathbf z,\boldsymbol\xi)],
\end{equation*}
where $\boldsymbol\xi:\Omega\to\Xi\subseteq\mathbb R^\ell$ is a random vector defined on a probability space $(\Omega,\mathcal F,\mathbb P)$. 
For convenience, we refer to problem~\eqref{pb:multi} as the population problem. 
Let $\{\xi_j\}_{j\ge1}$ be an i.i.d.\ sequence with the same distribution as $\boldsymbol\xi$.
For a fixed sample size $N$, we define
\begin{equation}
	\label{pb:saa_mul}
	F_i^N(\mathbf z):=
	\frac1N\sum_{j=1}^N f_i(\mathbf z,\xi_j).
\end{equation}
Throughout this paper, we make the following blanket assumptions. 

\begin{assumption}
	\label{ass:standing}
	The following conditions hold.
	\begin{enumerate}[(i)]
		\item For each $i\in[m]$ and each $\mathbf z\in\mathcal X$, $f_i:\mathcal X \times \Xi \to \mathbb R$ is measurable in $\xi$.
		
		\item For each $i\in[m]$ and for $\mathbb P$-a.e. $\xi\in\Xi$, the mapping $\mathbf z\mapsto f_i(\mathbf z,\xi)$ is continuously differentiable on an open set containing $\mathcal X$, and has $L_i(\xi)$-Lipschitz continuous gradient on $\mathcal X$, where $L_i(\xi)\ge0$ is measurable and
		$
		\mathbb E[L_i(\boldsymbol\xi)]<\infty.
		$
		
		\item For every nonempty compact set $S\subseteq\mathcal X$ and each $i\in[m]$, there exist integrable functions $K_{i,S},M_{i,S}:\Xi\to[0,\infty)$ such that, for $\mathbb P$-a.e. $\xi$,
		\begin{equation*}
			|f_i(\mathbf z,\xi)|\le K_{i,S}(\xi),
			\qquad
			\|\nabla f_i(\mathbf z,\xi)\|\le M_{i,S}(\xi),
			\qquad
			\forall\mathbf z\in S.
		\end{equation*}
		
		\item With probability one, for every $N\ge1$ and $i\in[m]$, each $F_i^N$ is bounded below and level-bounded on $\mathcal X$.
	\end{enumerate}
\end{assumption}

Assumption~\ref{ass:standing} implies that each $F_i$ has a Lipschitz continuous gradient on $\mathcal X$ with constant $L_i:=\mathbb E[L_i(\boldsymbol\xi)]$.
For every fixed $N$, the SAA objective $F_i^N$ is continuously differentiable on $\mathcal X$.
With probability one, its gradient has Lipschitz constant
\begin{equation*}
	L_i^N:=\frac1N\sum_{j=1}^N L_i(\xi_j),
	\qquad i\in[m].
\end{equation*}

Problem~\eqref{pb:multi} combines the difficulty of conflicting objectives with the absence of exact objective information.
Two broad strategies are commonly used.
The first reduces the vector problem to a stochastic single-objective problem through scalarization.
For example, the weighted-sum method selects $w  \in \Delta_m:=\{w\in\mathbb R^m: w\ge0, \mathbf 1^\top w=1\}$ and solves
\begin{equation*}
	\min_{\mathbf z\in\mathcal X}\; \phi_w(\mathbf z)
	:=\sum_{i=1}^m w_i\,F_i(\mathbf z).
\end{equation*}
The $\varepsilon$-constraint method instead fixes $\varepsilon_i\in\mathbb R$, $i=2,\ldots,m$, and solves
\begin{equation*}
	\min_{\mathbf z\in\mathcal X}\; F_1(\mathbf z)
	\quad\st\quad
	F_i(\mathbf z)\le\varepsilon_i,
	\quad i=2,\ldots,m.
\end{equation*}
Chebyshev-type scalarizations can additionally reach nonconvex regions of the Pareto front.
Approximating a broad portion of the Pareto set nevertheless requires solving many scalarized subproblems with different weights or thresholds. 
The second strategy replaces the expectations by sample averages and then applies a deterministic multi-objective method \citep{FliegeXu2011,shapiro2021lectures}.
For the i.i.d. sample $\{\xi_j\}_{j=1}^N$ in \eqref{pb:saa_mul}, the resulting sample-average approximation (SAA) problem is
\begin{equation}
	\label{pb:det_moo}
	\min_{\mathbf z\in\mathcal X} \; F^N(\mathbf z):=
	\big(F_1^N(\mathbf z),\ldots,F_m^N(\mathbf z)\big)^\top.
\end{equation}

For unconstrained deterministic multi-objective optimization, the multi-steepest descent method of Fliege and Svaiter \citep{Fliege2000} seeks a common descent direction at iterate $k$ by solving
\begin{equation}
	\label{subpb:smd}
	\mathbf d_k
	=\arg\min_{\mathbf d\in\mathbb R^n}
	\left\{\max_{i\in[m]}\nabla F_i^N(\mathbf z_k)^\top\mathbf d
	+\frac12\|\mathbf d\|^2\right\}.
\end{equation}
Introducing a slack variable $\delta$ converts the maximum term in subproblem~\eqref{subpb:smd} into linear constraints and yields the equivalent convex quadratic program
\begin{equation*}
	\min_{\delta\in\mathbb R,\,\mathbf d\in\mathbb R^n}
	\ \delta+\frac12\|\mathbf d\|^2
	\quad\st\quad
	\nabla F_i^N(\mathbf z_k)^\top\mathbf d\le\delta,
	\quad i=1,\ldots,m.
\end{equation*}
Its dual formulation is
\begin{equation}
	\lambda_k\in\arg\min_{\lambda\in\Delta_m}
	\frac12\left\|\sum_{i=1}^m\lambda_i\nabla F_i^N(\mathbf z_k)\right\|^2,
	\label{subpb:dual}
\end{equation}
and the associated descent direction is
\begin{equation*}
	\mathbf d_k=-\sum_{i=1}^m\lambda_{k,i}\nabla F_i^N(\mathbf z_k).
\end{equation*}
The direction $\mathbf d_k$ is therefore the negative minimum-norm convex combination of the objective gradients, which is the cornerstone of the multi-gradient descent algorithm (MGDA) \citep{MGDA}.

After computing $\mathbf d_k$, a standard update has the form $\mathbf z_{k+1}=\mathbf z_k+\alpha_k\mathbf d_k$ with the stepsize $\alpha_k$.
A backtracking Armijo condition adapted to vector-valued objectives \citep{steepestdescent} is often used to select $\alpha_k>0$  by requiring
\begin{equation}
	F^{N}(\mathbf{z}_k + \alpha_k \mathbf{d}_k) \preceq F^{N}(\mathbf{z}_k) + \gamma \alpha_k J_{F^{N}}(\mathbf{z}_k) \mathbf{d}_k,
	\label{armijo_line_search}
\end{equation}
where $\gamma\in(0,1)$ and
\begin{align*}
	J_{F^{N}} (\mathbf{z}) =
	\begin{bmatrix}
		[\nabla F^{N}_1(\mathbf{z})]^\top \\
		\vdots\\
		[\nabla F^{N}_m(\mathbf{z})]^\top
	\end{bmatrix}
	\in \mathbb{R}^{m \times n}.
\end{align*}
For nonconvex objectives, a monotone line search can be unnecessarily restrictive.
Nonmonotone variants permit temporary increases in selected components and use conditions of the form \citep{nonmonotone}
\begin{equation*}
	F^N(\mathbf z_k+\alpha_k\mathbf d_k)
	\preceq
	C^k+\gamma\alpha_kJ_{F^N}(\mathbf z_k)\mathbf d_k,
	\qquad C^k\succeq F^N(\mathbf z_k).
\end{equation*}
For example, a max-type method uses
\begin{equation*}
	C_i^k:=\max_{0\le j\le J(k)}F_i^N(\mathbf z_{k-j}),
	\qquad i\in[m],
\end{equation*}
where $J(k)$ is a dynamic window size, whereas an average-type method uses
\begin{equation*}
	C^k:=\frac{\eta q_{k-1}}{q_k}C^{k-1}+\frac1{q_k}F^N(\mathbf z_k),
	\qquad q_k:=\eta q_{k-1}+1.
\end{equation*}

The repeated objective evaluations required by line-search procedures are usually costly in stochastic settings.
Gratton, Jerad, and Toint \citep{Gratton2025SOFFAR} observe that, even for stochastic single-objective optimization, reliable function estimates may require substantially larger samples than gradient estimates.
Using the gradient mini-batch in a multi-objective Armijo test can therefore make the decision unstable, whereas increasing the batch size for function estimation raises the cost of every trial step.

Objective-function-free optimization (OFFO) methods avoid this difficulty by relying only on gradient information \citep{Gratton2025Complexity,Gratton2025SOFFAR,wang2026adaptive}.
In particular, OFFOAR1 \citep{Gratton2025SOFFAR} computes the update step $s_k$ by minimizing the regularized linear model
\begin{equation*}
	m_k(s):=\langle g_k,s\rangle+\frac{\sigma_k}{2}\|s\|^2,
\end{equation*}
where $g_k$ is a stochastic gradient estimator and $\sigma_k > 0$ is a regularization parameter. 
Then OFFOAR1 updates
\begin{equation*}
	\mathbf z_{k+1}=\mathbf z_k+s_k,
	\qquad
	\sigma_{k+1}=\sigma_k+\sigma_k\|s_k\|^2.
\end{equation*}
Under suitable conditions on the gradient estimator, OFFO-type methods attain optimal complexity.
They remain fully stochastic and require no function-value evaluations.

Motivated by this adaptive-regularization principle, we develop a projection-based method for constrained stochastic multi-objective optimization. 
This method computes each step from a regularized multi-gradient subproblem and updates its regularization parameter without a line-search procedure or a comparison of objective values.
The contributions of this paper can be summarized as follows. 
\begin{itemize}
	
	\item We introduce an adaptive projected-gradient method for the SAA formulation of constrained stochastic multi-objective optimization.
	Each iteration computes a feasible regularized multi-gradient step.
	The regularization parameter adapts to the step length, while the Tikhonov parameter regularizes the simplex multiplier and makes its selection unique and stable.
	
	\item We establish convergence and complexity results for the SAA problem.
	Under Assumption~\ref{ass:standing}, we prove that the iterates remain in a compact set and every accumulation point is Pareto stationary for problem~\eqref{pb:det_moo}. 
	The same analysis gives an explicit $\mathcal O(\varepsilon^{-2})$ complexity bound.
	
	\item We quantify the SAA approximation error.
	Compact-uniform convergence of the gradients implies consistency of the Pareto-stationarity residual and transfers suitably accurate SAA outputs to Pareto stationary points of the population problem.
	A covering-number argument further provides a high-probability finite-sample bound on compact sets.
	
	\item 
	Experiments on synthetic problems and real-world applications assess stationarity, feasibility, objective trade-offs, and computational cost of our algorithm.
	
\end{itemize}

The remainder of the paper is organized as follows. 
Section~\ref{sec:preliminaries} introduces notation and stationarity concepts.
Section~\ref{sec:A:gradients} develops the projected regularized multi-gradient algorithm, and Section~\ref{sec:C:complexity} presents its convergence analysis. 
Section~\ref{sec:D:consistency} studies consistency with the population problem and finite-sample residual bounds. 
Section~\ref{sec:experiment} reports the numerical experiments, followed by concluding remarks in the last section.

\section{Preliminaries}

\label{sec:preliminaries}

In this section, we introduce the basic notations and stationarity concepts used throughout this paper.

\subsection{Notations}

We write $\mathbb R$, $\mathbb R_+$, and $\mathbb R_{++}$ for the real, nonnegative real, and positive real numbers, respectively. 
The symbol $\|\cdot\|$ denotes the Euclidean norm for vectors, while $\|\cdot\|_F$ and $\|\cdot\|_2$ denote the Frobenius and spectral norms for matrices. 
For a nonempty closed convex set $C$, $\operatorname{Proj}_C$ is the Euclidean projection onto $C$. 
For $\mathbf z\in\mathcal X$, the normal cone to $\mathcal X$ at $\mathbf z$ is $\mathcal N_{\mathcal X}(\mathbf z) := \{v\in\mathbb R^n:\ v^\top(\mathbf y-\mathbf z) \leq 0, \forall\mathbf y\in\mathcal X\}$.

\subsection{Stationarity conditions}

We now recall the optimality and stationarity conditions for problems \eqref{pb:multi} and \eqref{pb:det_moo}.
For $u, v \in \mathbb R^m$, the relation $u \preceq v$ means that $u_j \leq v_j$ for every $j \in [m]$, while $u \prec v$ means that all these inequalities are strict.

\begin{definition}[{\citealp[p. 38]{ehrgott2005multicriteria}}]
	\label{def:pareto_op}
	A point $\mathbf z^\ast \in \mathcal{X}$ is called Pareto optimal for problem~\eqref{pb:multi} if there is no $\mathbf z \in \mathcal{X}$ such that
	\begin{equation*}
		F(\mathbf z) \preceq F(\mathbf z^\ast)
		\quad\text{and}\quad
		F(\mathbf z) \neq F(\mathbf z^\ast).
	\end{equation*}
\end{definition}

\begin{definition}[{\citealp[p. 19]{nonlinearMOO}}]
	\label{def:weak_pareto_op}
	A point $\mathbf z^\ast\in \mathcal{X}$ is called weakly Pareto optimal for problem~\eqref{pb:multi} if there is no $\mathbf z\in \mathcal{X}$ such that
	\begin{equation*}
		F(\mathbf z)\prec F(\mathbf z^\ast).
	\end{equation*}
\end{definition}

%

Clearly, every Pareto optimal point is weakly Pareto optimal, while the converse does not hold in general. 
For nonconvex multi-objective optimization problems, computing a Pareto optimal or weakly Pareto optimal point is generally challenging, since these notions require global comparisons over the feasible set. 
In contrast, first-order stationarity conditions provide tractable necessary conditions for weak Pareto optimality and are therefore more amenable to algorithmic analysis. 
This motivates the following notion of Pareto stationarity. 
For convenience, we write
\begin{equation*}
	G(\mathbf z):=
	\big[\nabla F_1(\mathbf z),\ldots,\nabla F_m(\mathbf z)\big],
	\
	G^N(\mathbf z):=
	\big[\nabla F_1^N(\mathbf z),\ldots,\nabla F_m^N(\mathbf z)\big].
\end{equation*}

\begin{definition}[\citealp{cond_PS}]
	\label{def:constrained-pareto-stat}
	A point $\mathbf z^\ast\in\mathcal X$ is called Pareto stationary for problem~\eqref{pb:multi} if
	\begin{equation}
		\label{eq:normal-cps}
		0\in G(\mathbf z^\ast)\Delta_m+\mathcal N_{\mathcal X}(\mathbf z^\ast).
	\end{equation}
\end{definition}



The corresponding notions for the SAA problem can be defined analogously. 
Specifically, Pareto optimality, weak Pareto optimality, and Pareto stationarity for problem~\eqref{pb:det_moo} are obtained by replacing $F$ and $G$ in Definitions \ref{def:pareto_op}, \ref{def:weak_pareto_op}, and \ref{def:constrained-pareto-stat} with their SAA counterparts $F^N$ and $G^N$, respectively. 
For brevity, we omit the details.

\begin{lemma}
	\label{lem:existence-pareto}
	Under Assumption~\ref{ass:standing}, problem~\eqref{pb:det_moo} admits at least one Pareto optimal point. 
\end{lemma}

\begin{proof}
	The proof is straightforward and is omitted here.
\end{proof}

\begin{theorem}[{\citealp[Lemma 1.3]{constrained_PS}}]
	\label{thm:A:necessary}
	If $\mathbf z^\ast \in\mathcal X$ is a weakly Pareto optimal point of problem~\eqref{pb:det_moo}, then $\mathbf z^\ast$ is Pareto stationary for problem~\eqref{pb:det_moo}.
\end{theorem}

\subsection{Pareto-stationarity residual}

\label{sec:A:constr-stat}

For $\mathbf z \in\mathcal X$ and $H \in \mathbb R^{n\times m}$, we define
\begin{equation*}
	\vartheta_{\mathcal X}(\mathbf z;H)
	:=
	\operatorname{dist}
	\left(
	0,\,
	H\Delta_m+\mathcal N_{\mathcal X}(\mathbf z)
	\right).
\end{equation*}
Then the population Pareto-stationarity residual can be represented by
\begin{equation}
	\label{eq:A:thetaX}
	\Theta_{\mathcal X}(\mathbf z)
	:=
	\vartheta_{\mathcal X}(\mathbf z;G(\mathbf z))
	= 
	\operatorname{dist}
	\left(
	0,\,
	G(\mathbf z)\Delta_m+\mathcal N_{\mathcal X}(\mathbf z)
	\right).
\end{equation}
Moreover, the SAA Pareto-stationarity residual is
\begin{equation}
	\label{eq:A:thetaXN}
	\Theta_{\mathcal X}^N(\mathbf z)
	:=
	\vartheta_{\mathcal X}(\mathbf z;G^N(\mathbf z))
	=
	\operatorname{dist}
	\left(
	0,\,
	G^N(\mathbf z)\Delta_m+\mathcal N_{\mathcal X}(\mathbf z)
	\right).
\end{equation}
The equality $\Theta_{\mathcal X}(\mathbf z)=0$ holds exactly when some $\lambda\in\Delta_m$ satisfies
\begin{equation*}
	-G(\mathbf z)\lambda\in \mathcal N_{\mathcal X}(\mathbf z),
\end{equation*}
which characterizes Pareto stationarity for problem~\eqref{pb:multi}. 
Similarly, the condition $\Theta_{\mathcal X}^N(\mathbf z)=0$ gives the corresponding characterization for the SAA problem~\eqref{pb:det_moo}.

\section{Algorithm Design}

\label{sec:A:gradients}

In this section, we devise a projected regularized multi-gradient algorithm for solving the SAA problem~\eqref{pb:det_moo}. 
The proposed method is fully adaptive and enjoys convergence guarantees without requiring prior knowledge of any problem-specific constants or parameters.


\subsection{Projected regularized multi-gradient step}

\label{sec:A:reg}

We first construct a feasible multi-gradient step that will be used in the proposed algorithm. 
Recall from Subsection~\ref{sec:A:constr-stat} that Pareto stationarity at a point $\mathbf z^\ast \in \mathcal X$ is characterized by the existence of a multiplier $\lambda \in \Delta_m$ such that
\begin{equation*}
	- G (\mathbf z^\ast) \lambda \in \mathcal N_{\mathcal X}(\mathbf z^\ast).
\end{equation*}
Our goal is therefore to construct, from a current point $\mathbf z \in \mathcal X$, a feasible displacement $\mathbf p$ that provides an approximate relation of this form at the new point $\mathbf z + \mathbf p$.

For notational simplicity, we use $H \in \mathbb R^{n \times m}$ to denote a generic gradient matrix in the following discussion. 
To preserve feasibility, we define the displacement set
\begin{equation*}
	\mathcal D (\mathbf z)
	:= \mathcal X - \mathbf z
	= \{\mathbf p \in \mathbb R^n : \mathbf z + \mathbf p \in \mathcal X\}.
\end{equation*}
Then we consider the following projected regularized multi-gradient subproblem,
\begin{equation}
	\label{eq:A:proj-reg-cauchy}
	\mathbf p_{\rho,\sigma}(\mathbf z;H)
	:=
	\arg\min_{\mathbf p\in\mathcal D(\mathbf z)}
	\left\{
	\frac{\sigma}{2}\|\mathbf p\|^2
	+
	\max_{\lambda\in\Delta_m}
	\left[
	\langle H\lambda,\mathbf p\rangle
	-
	\frac{\rho}{2\sigma}
	\left(\|\lambda\|^2-\frac1m\right)
	\right]
	\right\},
\end{equation}
where $\sigma > 0$ and $\rho > 0$ are two regularization parameters. 
The shift $1/m$ normalizes the inner maximization so that its value at $\mathbf p=0$ is zero.

In subproblem~\eqref{eq:A:proj-reg-cauchy}, two regularization parameters play different roles. 
The parameter $\sigma$ controls the magnitude of the resulting displacement, whereas $\rho$ regularizes the convex combination of the gradients. 
In particular, the Tikhonov regularization term in the inner maximization problem with respect to $\lambda$ provides a unique and stable multiplier selection. 
Specifically, we define
\begin{equation*}
	\Lambda_{\rho,\sigma}(\mathbf p;H)
	:=
	\arg\max_{\lambda\in\Delta_m}
	\left[
	\langle H\lambda,\mathbf p\rangle
	-
	\frac{\rho}{2\sigma}
	\left(\|\lambda\|^2-\frac1m\right)
	\right],
\end{equation*}
for fixed $\mathbf p$. 
Completing the square gives
\begin{equation}
	\label{eq:inner-lambda-proj}
	\Lambda_{\rho,\sigma}(\mathbf p;H)
	=
	\operatorname{Proj}_{\Delta_m}
	\left(
	\frac{\sigma}{\rho}H^\top \mathbf p
	\right).
\end{equation}
Thus, the multiplier is obtained by a projection onto $\Delta_m$.

Let
\begin{equation*}
	\psi_{\rho,\sigma}(u)
	:=
	\max_{\lambda\in\Delta_m}
	\left[
	\langle \lambda,u\rangle
	-
	\frac{\rho}{2\sigma}
	\left(\|\lambda\|^2-\frac1m\right)
	\right],
	\quad u\in\mathbb R^m.
\end{equation*}
Then the objective function in \eqref{eq:A:proj-reg-cauchy} can be written as
\begin{equation*}
	\Psi_{\rho,\sigma}(\mathbf p;\mathbf z,H)
	:=
	\frac{\sigma}{2}\|\mathbf p\|^2+\psi_{\rho,\sigma}(H^\top \mathbf p),
	\quad \mathbf p\in\mathcal D(\mathbf z).
\end{equation*}
By Danskin's theorem and \eqref{eq:inner-lambda-proj}, we have
\begin{equation*}
	\nabla_\mathbf p\Psi_{\rho,\sigma}(\mathbf p;\mathbf z,H)
	=
	\sigma \mathbf p
	+
	H\Lambda_{\rho,\sigma}(\mathbf p;H).
\end{equation*}
The function $\Psi_{\rho,\sigma}(\cdot;\mathbf z,H)$ is strongly convex in $\mathbf p$. 
Since $\mathcal D(\mathbf z)$ is nonempty, closed, and convex, subproblem~\eqref{eq:A:proj-reg-cauchy} has the unique solution $\mathbf p_{\rho,\sigma}(\mathbf z;H)$.

According to the optimality condition of subproblem \eqref{eq:A:proj-reg-cauchy}, we can obtain that
\begin{equation*}
	0\in
	\sigma \mathbf p
	+
	H\Lambda_{\rho,\sigma}(\mathbf p;H)
	+
	\mathcal N_{\mathcal D(\mathbf z)}(\mathbf p),
\end{equation*}
for $\mathbf p=\mathbf p_{\rho,\sigma}(\mathbf z;H)$. 
Since $\mathcal D(\mathbf z)=\mathcal X-\mathbf z$, we have $\mathcal N_{\mathcal D(\mathbf z)}(\mathbf p)=\mathcal N_{\mathcal X}(\mathbf z+\mathbf p)$. 
It then follows that
\begin{equation}
	\label{eq:normal-certificate-subproblem}
	-H\lambda-\sigma \mathbf p_{\rho,\sigma}(\mathbf z;H)
	\in
	\mathcal N_{\mathcal X}(\mathbf z+\mathbf p_{\rho,\sigma}(\mathbf z;H)),
\end{equation}
for $\lambda:=\Lambda_{\rho,\sigma}(\mathbf p_{\rho,\sigma}(\mathbf z;H);H)$.

When $\mathcal X=\mathbb R^n$, we have $\mathcal D(\mathbf z)=\mathbb R^n$.
In this case, subproblem~\eqref{eq:A:proj-reg-cauchy} admits the following closed-form solution,
\begin{equation*}
	\mathbf p_{\rho,\sigma}(\mathbf z;H)
	=
	-\frac1\sigma H\lambda_\rho(H),
\end{equation*}
where
\begin{equation*}
	\lambda_\rho(H)
	=
	\arg\min_{\lambda\in\Delta_m}
	\left\{
	\frac12\|H\lambda\|^2+\frac{\rho}{2}\|\lambda\|^2
	\right\}.
\end{equation*}
As a result, the algorithmic update becomes $\mathbf z^+ = \mathbf z+\mathbf p_{\rho,\sigma}(\mathbf z;H) = \mathbf z-H\lambda_\rho(H)/\sigma$, which is the regularized multi-gradient step with the stepsize $1 / \sigma$.


\begin{remark}
	\label{rem:compute-projected-cauchy}
	The cost of computing $\mathbf p_{\rho,\sigma}(\mathbf z;H)$ depends on the structure of $\mathcal X$, which is inexpensive when $\mathcal X$ is a box, a simplex, a Euclidean ball, or a structured set $\mathcal X$ with tractable projection. 
	Two standard first-order approaches are described below.

	The first one is the projected-gradient algorithm \citep{condat_l1ball,proximal_alg}, whose update scheme can be expressed as
	\begin{equation*}
		\begin{aligned}
			\mathbf p^{r+1}
			=
			\operatorname{Proj}_{\mathcal D(\mathbf z)}
			\biggl[
			\mathbf p^r
			-
			\alpha_r\biggl(
			\sigma \mathbf p^r
			&+
			H\operatorname{Proj}_{\Delta_m}
			\biggl(
			\frac{\sigma}{\rho}H^\top \mathbf p^r
			\biggr)
			\biggr)
			\biggr].
		\end{aligned}
	\end{equation*}
	Nonexpansiveness of Euclidean projections implies that the gradient of the objective function has Lipschitz constant at most
	\begin{equation*}
		L_{\mathbf p}:=\sigma+\frac{\sigma}{\rho}\|H\|_2^2.
	\end{equation*}
	Thus, the standard convergence guarantee of projected-gradient algorithms applies, for example, when $0<\alpha_r<2/L_{\mathbf p}$.
	Since $\operatorname{Proj}_{\mathcal D(\mathbf z)}(q)=\operatorname{Proj}_{\mathcal X}(\mathbf z+q)-\mathbf z$, this algorithm only requires projections onto $\mathcal X$ and the simplex $\Delta_m$.

	Subproblem~\eqref{eq:A:proj-reg-cauchy} also has a saddle-point form and can be treated by primal-dual splitting methods \citep{proj_method_for_subprob}.
	Let $\iota_C$ denote the indicator function of a set $C$, which is zero on $C$ and $+\infty$ otherwise. 
	We define
	\begin{equation*}
		f(\mathbf p):=\frac{\sigma}{2}\|\mathbf p\|^2+\iota_{\mathcal D(\mathbf z)}(\mathbf p),
		\qquad
		g(\lambda):=\frac{\rho}{2\sigma}
		\left(\|\lambda\|^2-\frac1m\right)+\iota_{\Delta_m}(\lambda).
	\end{equation*}
	With stepsizes $\tau_{\mathbf p},\tau_\lambda>0$ satisfying $\tau_{\mathbf p}\tau_\lambda\|H\|_2^2<1$, we can use the following method to solve subproblem~\eqref{eq:A:proj-reg-cauchy},
	\begin{equation*}
		\left\{
		\begin{aligned}
			\lambda^{r+1}
			& =
			\operatorname{prox}_{\tau_\lambda g}
			\left(
			\lambda^r+\tau_\lambda H^\top \bar{\mathbf p}^{r}
			\right),\\
			\mathbf p^{r+1}
			& =
			\operatorname{prox}_{\tau_{\mathbf p} f}
			\left(
			\mathbf p^r-\tau_{\mathbf p} H\lambda^{r+1}
			\right),\\
			\bar{\mathbf p}^{r+1}
			& =
			\mathbf p^{r+1}
			+
			\theta(\mathbf p^{r+1}-\mathbf p^r),
		\end{aligned}
		\right.
	\end{equation*}
	where $\theta\in[0,1]$ is a constant. 
\end{remark}

\subsection{Algorithm description}

\label{sec:C:algorithm-subsection}

We now combine the projected regularized multi-gradient step with an adaptive scheme for updating parameters. 
The resulting method is summarized in Algorithm~\ref{alg:projected-SAA-RMGDA}, which preserves feasibility at every iteration.

\begin{algorithm}[htbp]
	\caption{SAA Regularized Multi-Gradient Descent Algorithm (SAA-RMGDA)}
	\label{alg:projected-SAA-RMGDA}
	\begin{algorithmic}
		\REQUIRE Initial point $\mathbf z_0\in\mathcal X$; initial regularization parameter $\sigma_0>0$; Tikhonov parameters $\{\rho_k\}_{k\ge0}\subseteq\mathbb R_{++}$ with $\sum_{k = 0}^{\infty} \rho_k < \infty$.
		\STATE Set $k:=0$. 
		\WHILE{stopping criterion not met}
		\STATE \textbf{(SAA gradients)} For each $i\in[m]$, set
		$
		g_i^N(\mathbf z_k):=\nabla F_i^N(\mathbf z_k),
		$
		and form
		$
		G_k^N
		:=
		G^N(\mathbf z_k)
		=
		[g_1^N(\mathbf z_k),\ldots,g_m^N(\mathbf z_k)].
		$
		
		\STATE \textbf{(Projected regularized multi-gradient step)} Compute
		\begin{equation}
			\label{subp:pk}
			\mathbf p_k
			\in
			\arg\min_{\mathbf p\in\mathcal D(\mathbf z_k)}
			\left\{
			\frac{\sigma_k}{2}\|\mathbf p\|^2
			+
			\max_{\lambda\in\Delta_m}
			\left[
			\langle G_k^N\lambda,\mathbf p\rangle
			-
			\frac{\rho_k}{2\sigma_k}
			\left(\|\lambda\|^2-\frac1m\right)
			\right]
			\right\}.
		\end{equation}
		
		\STATE \textbf{(Update step)} Set
		$
		\mathbf z_{k+1}:=\mathbf z_k+\mathbf p_k.
		$
		
		\STATE \textbf{(Regularization parameter)} Set
		$
		\sigma_{k+1}:=\sigma_k+\sigma_k\|\mathbf p_k\|^2.
		$
		
		\STATE $k:=k+1$.
		\ENDWHILE
	\end{algorithmic}
\end{algorithm}


A key feature of the proposed method is that the regularization parameter is updated adaptively using information generated by the current iterate, without requiring any prior knowledge of problem-dependent constants such as the Lipschitz constants. 
Specifically, after computing the projected regularized multi-gradient step $\mathbf p_k$, the parameter $\sigma_k$ is updated according to $\sigma_{k+1}:=\sigma_k+\sigma_k\|\mathbf p_k\|^2$. 
Hence, the amount of regularization is automatically adjusted according to the magnitude of the computed step. 
A larger step produces a larger increase in $\sigma_k$, whereas a small step results in only a mild adjustment. 
This mechanism eliminates the need for line-search procedures or prescribed stepsizes. 
The auxiliary Tikhonov parameter $\rho_k$ follows an arbitrary positive-summable sequence and is used to regularize the simplex multiplier.
Based on these constructions, Algorithm~\ref{alg:projected-SAA-RMGDA} enjoys the convergence and complexity guarantees, which will be established in the next section.

\section{Convergence Analysis}

\label{sec:C:complexity}

In this section, we demonstrate convergence properties of our algorithm. 
It is shown that every accumulation point is Pareto stationary for problem~\eqref{pb:det_moo}. 
We also establish its iteration complexity.

Let $\lambda_k \in \Delta_m$ be the unique maximizer in the inner problem of \eqref{subp:pk} for every $k \geq 0$. 
Since $\Delta_m$ is closed and compact, the sequence $\{\lambda_k\}$ has at least one accumulation point. 
To analyze the behavior of iterates, we choose an arbitrary accumulation point $\hat{\lambda} \in \Delta_m$ of $\{\lambda_k\}$ and introduce a scalar function
\begin{equation*}
	\Phi^N(\mathbf z):=\sum_{i=1}^m \hat{\lambda}_i F_i^N(\mathbf z).
\end{equation*}
Assumption~\ref{ass:standing}(iv) makes $\Phi^N$ both bounded below and level-bounded on $\mathcal X$. 
Let $\Phi_{\rm low}^N$ be any lower bound of $\Phi^N$ on $\mathcal X$. 
Moreover, we define
\begin{equation*}
	S_\infty := \sum_{k = 0}^{\infty} \rho_k < \infty.
\end{equation*}

\begin{lemma}
	\label{lem:C:constrained-descent}
	Let $\{\mathbf z_k\}$ be a sequence generated by Algorithm~\ref{alg:projected-SAA-RMGDA}. 
	Then, for all $k \geq 0$, we have
	\begin{equation}
		\label{eq:C:potential-descent}
		\Phi^N (\mathbf z_k) - \Phi^N (\mathbf z_{k + 1})
		\geq
		\frac{1}{2} \left(\sigma_k - L_\Phi^N \right)  \|\mathbf p_k\|^2
		- \frac{1}{\sigma_0} \kappa_m \rho_k,
	\end{equation}
	where $L_\Phi^N := \sum_{i = 1}^{m} \hat{\lambda}_i L_i^N$ and $\kappa_m := (1 - 1 / m) / 2$ are two constants. 
\end{lemma}

\begin{proof}
	In subproblem~\eqref{subp:pk}, $\mathbf p = 0$ is feasible and the objective value at $\mathbf p = 0$ is zero. 
	Let $e_i$ be the $i$-th standard basis vector in $\mathbb{R}^m$. 
	Then the optimality of $\mathbf p_k$ and evaluation of the inner maximum at $\lambda = e_i$ give
	\begin{equation*}
		\nabla F_i^N (\mathbf z_k)^\top \mathbf p_k
		\leq
		-\frac{1}{2} \sigma_k \|\mathbf p_k\|^2
		+ \frac{1}{\sigma_k}\kappa_m \rho_k,
	\end{equation*}
	for every $i \in [m]$. 
	Taking the weighted sum and using the $L_\Phi^N$-smoothness of $\Phi^N$ yields
	\begin{equation*}
		\Phi^N(\mathbf z_k)-\Phi^N(\mathbf z_{k+1})
		\ge
		\frac{1}{2} \left(\sigma_k - L_\Phi^N \right)  \|\mathbf p_k\|^2
		-\frac{1}{\sigma_k} \kappa_m \rho_k.
	\end{equation*}
	Since $\sigma_k\ge\sigma_0$, we finally obtain \eqref{eq:C:potential-descent}.
\end{proof}

\begin{lemma}
	\label{lem:bd-sigma}
	There exists a constant $\Sigma^N > 0$ such that
	\begin{equation*}
		\sigma_k \leq \Sigma^N,
	\end{equation*}
	for all $k \geq 0$. 
\end{lemma}

\begin{proof}
	Let $\bar{k} := \inf \{k \geq 0 \mid \sigma_k \geq 2 L_\Phi^N\}$. 
	If $\bar{k}$ does not exist, we have $\sigma_k < 2 L_\Phi^N$ for all $k \geq 0$. 
	Then the sequence $\{\sigma_k\}$ is already bounded. 
	Hence, we next consider the scenario where $\bar{k} <  \infty$. 
	Since the sequence $\{\sigma_k\}$ is nondecreasing, it holds that $\sigma_k \geq 2 L_\Phi^N$ for all $k \geq \bar{k}$. 
	As a result, we have
	\begin{equation*}
		\frac{1}{2} \left(\sigma_k - L_\Phi^N \right) \geq \frac{1}{4} \sigma_k,
	\end{equation*}
	for all $k \geq \bar{k}$. 
	Then it follows from \eqref{eq:C:potential-descent} that
	\begin{equation*}
		\Phi^N (\mathbf z_k) - \Phi^N (\mathbf z_{k + 1})
		\geq
		\frac{1}{4} \sigma_k \|\mathbf p_k\|^2
		- \frac{1}{\sigma_0} \kappa_m \rho_k
		=
		\frac{1}{4} (\sigma_{k + 1} - \sigma_k)
		- \frac{1}{\sigma_0} \kappa_m \rho_k.
	\end{equation*}
	Summing the above relationship over $k$ from $\bar{k}$ to $K - 1$, we arrive at
	\begin{equation*}
		\frac{1}{4} (\sigma_{K} - \sigma_{\bar{k}})
		\leq \Phi^N (\mathbf z_{\bar{k}}) - \Phi_{\rm low}^N
		+ \frac{\kappa_m}{\sigma_0} \sum_{k = \bar{k}}^{K - 1} \rho_k
		\leq \Phi^N (\mathbf z_{\bar{k}}) - \Phi_{\rm low}^N
		+ \frac{\kappa_m}{\sigma_0} S_\infty,
	\end{equation*}
	which further implies that
	\begin{equation*}
		\sigma_{K} \leq \sigma_{\bar{k}} + 4 \left( \Phi^N (\mathbf z_{\bar{k}}) - \Phi_{\rm low}^N + \frac{\kappa_m}{\sigma_0} S_\infty \right),
	\end{equation*}
	for all $K \geq \bar{k}$. 
	Moreover, it holds that $\sigma_K < 2 L_\Phi^N$ for all $K < \bar{k}$.
	Consequently, we have
	\begin{equation*}
		\sigma_k \leq \Sigma^N := \max \left\{ 2 L_\Phi^N, \; \sigma_{\bar{k}} + 4 \left( \Phi^N (\mathbf z_{\bar{k}}) - \Phi_{\rm low}^N + \frac{\kappa_m}{\sigma_0} S_\infty \right) \right\},
	\end{equation*}
	for all $k \geq 0$. 
	The proof is completed. 
\end{proof}

\begin{theorem}
	\label{thm:C:constrained-residual-complexity}
	For all $K \geq 0$, it holds that
	\begin{equation}
		\label{eq:C:theta-complexity}
		\min_{0 \le t \le K} \Theta_{\mathcal X}^N (\mathbf z_{t+1})
		\le
		C^N \Sigma^N \sqrt{\frac{\Sigma^N - \sigma_0}{\sigma_0 (K + 1)}},
	\end{equation}
	where $C^N := 1 + L_G^N / \sigma_0$ and $L_G^N := (\sum_{i = 1}^m (L_i^N)^2)^{1/2}$. 
	Consequently, Algorithm~\ref{alg:projected-SAA-RMGDA} is capable of finding an iterate $\mathbf z_{t+1}$ that satisfies $\Theta_{\mathcal X}^N(\mathbf z_{t+1})\le\varepsilon$ after at most $\mathcal O (\varepsilon^{-2})$ iterations. 
\end{theorem}

\begin{proof}
	Since $\sigma_{K + 1} - \sigma_0 = \sum_{k = 0}^K (\sigma_{k + 1} - \sigma_k) = \sum_{k = 0}^K \sigma_k \|\mathbf p_k\|^2 \geq \sigma_0 \sum_{k = 0}^K \|\mathbf p_k\|^2$, we have
	\begin{equation}
		\label{eq:sum-pk}
		\sum_{k = 0}^K \left\| \mathbf p_k \right\|^2
		\leq \frac{\sigma_{K + 1} - \sigma_0}{\sigma_0}
		\leq \frac{\Sigma^N - \sigma_0}{\sigma_0},
	\end{equation}
	for all $K \geq 0$. 
	Applying the optimality condition~\eqref{eq:normal-certificate-subproblem} with $H=G^N(\mathbf z_k)$, $\rho=\rho_k$, and $\sigma=\sigma_k$ yields that
	\begin{equation}
		\label{eq:C:normal-certificate}
		-G^N(\mathbf z_k)\lambda_k-\sigma_k\mathbf p_k
		\in\mathcal N_{\mathcal X}(\mathbf z_{k+1}),
	\end{equation}
	which provides a feasible normal-cone certificate for $\Theta_{\mathcal X}^N(\mathbf z_{k+1})$. 
	Since $\|\lambda_k\|\le1$ and $G^N$ is $L_G^N$-Lipschitz in Frobenius norm, we have
	\begin{equation*}
		\Theta_{\mathcal X}^N(\mathbf z_{k+1})
		\le
		\left\|
		\bigl(G^N(\mathbf z_{k+1})-G^N(\mathbf z_k)\bigr)\lambda_k
		-\sigma_k\mathbf p_k
		\right\| 
		\le
		\bigl(L_G^N+\sigma_k\bigr)\|\mathbf p_k\|
		\le
		\left(1+\frac{L_G^N}{\sigma_0}\right) \sigma_k\|\mathbf p_k\|
		\leq C^N \Sigma^N \|\mathbf p_k\|.
	\end{equation*}
	Then it can be readily verified that
	\begin{equation*}
		\min_{0 \leq t \leq K} \left( \Theta_{\mathcal X}^N (\mathbf z_{t + 1}) \right)^2
		\leq 
		\frac{1}{K + 1} \sum_{t = 0}^K \left( \Theta_{\mathcal X}^N (\mathbf z_{t + 1}) \right)^2
		\leq
		\frac{(C^N \Sigma^N)^2}{K + 1} \sum_{t = 0}^K \left\|\mathbf p_t\right\|^2
		\leq 
		\frac{(C^N \Sigma^N)^2 (\Sigma^N - \sigma_0)}{\sigma_0 (K + 1)},
	\end{equation*}
	which proves~\eqref{eq:C:theta-complexity}.
\end{proof}


The $O (\varepsilon^{-2})$ complexity bound is established for the nonconvex setting. 
In the convex case, however, one can generally expect a more favorable complexity guarantee. 
For instance, the accelerated techniques developed in \cite{zhao2023accelerated,zhao2024accelerated} could potentially be incorporated to further improve the convergence rate. 
A systematic investigation along this direction is left for future work.

\begin{theorem}
	\label{thm:C:constrained-stationarity}
	The sequence $\{\mathbf z_k\}$ generated by Algorithm~\ref{alg:projected-SAA-RMGDA} has at least one accumulation point. 
	Moreover, every accumulation point of $\{\mathbf z_k\}$ is Pareto stationary for the SAA problem~\eqref{pb:det_moo}. 
\end{theorem}

\begin{proof}
	According to \eqref{eq:C:potential-descent}, we have
	\begin{equation*}
		\Phi^N (\mathbf z_{k + 1}) - \Phi^N (\mathbf z_k)
		\leq
		\frac{1}{2} \left(L_\Phi^N - \sigma_k\right)  \|\mathbf p_k\|^2
		+ \frac{\kappa_m}{\sigma_0} \rho_k
		\leq
		\frac{L_\Phi^N}{2} \|\mathbf p_k\|^2
		+ \frac{\kappa_m}{\sigma_0} \rho_k.
	\end{equation*}
	Summing the above relationship over $k$ from $0$ to $K - 1$ yields that
	\begin{equation*}
		\Phi^N (\mathbf z_K)
		\leq
		\Phi^N (\mathbf z_0)
		+ \frac{L_\Phi^N}{2} \sum_{k = 0}^{K - 1} \|\mathbf p_k\|^2
		+ \frac{\kappa_m}{\sigma_0} \sum_{k = 0}^{K - 1} \rho_k
		\leq
		\Phi^N (\mathbf z_0)
		+ \frac{L_\Phi^N (\Sigma^N - \sigma_0)}{2 \sigma_0}
		+ \frac{\kappa_m}{\sigma_0} S_\infty
		=: C,
	\end{equation*}
	where the last inequality follows from \eqref{eq:sum-pk} and the definition of $S_\infty$. 
	Hence, all iterates $\{\mathbf z_k\}$ generated by Algorithm~\ref{alg:projected-SAA-RMGDA} belong to the level set $\{\mathbf z\in\mathcal X:\Phi^N(\mathbf z)\le C\}$, which is compact as $\Phi^N$ is continuous and level-bounded on the closed set $\mathcal X$. 
	Then we know that $\{\mathbf z_k\}$ has at least one accumulation point.

	Let $\bar{\mathbf z}$ be an accumulation point of $\{\mathbf z_k\}$. 
	Upon taking $K \to \infty$ in \eqref{eq:sum-pk}, we arrive at $\sum_{k = 0}^\infty \|\mathbf p_k\|^2 < \infty$, which implies that $\mathbf p_k \to 0$ as $k \to \infty$. 
	Let $\{\mathbf z_{k_j}\}$ be the subsequence of $\{\mathbf z_k\}$ that converges to $\bar{\mathbf z}$. 
	Then we have $\mathbf z_{k_j + 1} = \mathbf z_{k_j} + \mathbf p_{k_j} \to \bar{\mathbf z}$ as $j \to \infty$.
	By the compactness of $\Delta_m$, after passing to a further subsequence, we assume that $\lambda_{k_j} \to \bar\lambda \in \Delta_m$ as $j \to \infty$.
	From~\eqref{eq:C:normal-certificate}, we can obtain that
	\begin{equation}
		\label{eq:normal-kj}
		-G^N(\mathbf z_{k_j})\lambda_{k_j}-\sigma_{k_j}\mathbf p_{k_j}
		\in\mathcal N_{\mathcal X}(\mathbf z_{k_j+1}).
	\end{equation}
	Lemma~\ref{lem:bd-sigma} guarantees that $\{\sigma_k\}$ is bounded. 
	Since $\mathbf p_{k_j} \to 0$ and $G^N$ is continuous, the left-hand side of \eqref{eq:normal-kj} converges to $-G^N(\bar{\mathbf z})\bar\lambda$ as $j \to \infty$. 
	The graph of the normal-cone mapping to a closed convex set is closed. 
	Hence, we have
	\begin{equation*}
		-G^N(\bar{\mathbf z})\bar\lambda
		\in\mathcal N_{\mathcal X}(\bar{\mathbf z}),
	\end{equation*}
	which indicates that $\bar{\mathbf z}$ is Pareto stationary for the SAA problem~\eqref{pb:det_moo}. 
\end{proof}

\section{Consistency of SAA}

\label{sec:D:consistency}

The consistency of SAA problems and of their stationary points has been extensively studied in stochastic programming and stochastic generalized equations \cite{ShapiroXu2007,ralph2011convergence,FliegeXu2011,shapiro2021lectures}. 
In this section, we specialize these consistency results to the Pareto-stationarity residual used in our analysis. 
We show that every accumulation point of increasingly accurate SAA-stationary outputs is Pareto stationary for problem~\eqref{pb:multi}, and derive an explicit finite-sample residual bound.

\subsection{Population Pareto stationarity}

\label{sec:D:population-consistency}

We first compare the SAA and population residuals through their gradient matrices.
Recall the auxiliary residual $\vartheta_{\mathcal X}$ from Subsection~\ref{sec:A:constr-stat}, for which
\begin{equation*}
	\Theta_{\mathcal X}(\mathbf z)=\vartheta_{\mathcal X}(\mathbf z;G(\mathbf z)),
	\qquad
	\Theta_{\mathcal X}^N(\mathbf z)=\vartheta_{\mathcal X}(\mathbf z;G^N(\mathbf z)).
\end{equation*}

\begin{lemma}
	\label{lem:D:thetaX-lip-H}
	For every $\mathbf z\in\mathcal X$ and every $H,\widehat H\in\mathbb R^{n\times m}$, we have
	\begin{equation*}
		\bigl|\vartheta_{\mathcal X}(\mathbf z;\widehat H)
		-\vartheta_{\mathcal X}(\mathbf z;H)\bigr|
		\le\|\widehat H-H\|_F.
	\end{equation*}
\end{lemma}

\begin{proof}
	For every $\lambda\in\Delta_m$ and $v\in\mathcal N_{\mathcal X}(\mathbf z)$,
	\begin{equation*}
		\|\widehat H\lambda+v\|
		\le
		\|H\lambda+v\|+\|\widehat H-H\|_F\|\lambda\|
		\le
		\|H\lambda+v\|+\|\widehat H-H\|_F.
	\end{equation*}
	Taking the infimum over $(\lambda,v)$ and then interchanging $H$ and $\widehat H$ proves the claim.
\end{proof}

\begin{theorem}
	\label{thm:D:thetaX-uniform}
	Under Assumption~\ref{ass:standing}, for every nonempty compact set $\mathcal K\subseteq\mathcal X$, we have
	\begin{equation*}
		\sup_{\mathbf z\in\mathcal K}
		\bigl|\Theta_{\mathcal X}^N(\mathbf z)-\Theta_{\mathcal X}(\mathbf z)\bigr|
		\longrightarrow0,
		\quad\text{as }N\to\infty,
	\end{equation*}
	with probability one. 
\end{theorem}

\begin{proof}
	For each $i\in[m]$, the compact-uniform strong law of large numbers for random functions \cite{ShapiroXu2007,shapiro2021lectures}, together with Assumption~\ref{ass:standing}(ii)--(iii), gives
	\begin{equation*}
		\sup_{\mathbf z\in\mathcal K}
		\|\nabla F_i^N(\mathbf z)-\nabla F_i(\mathbf z)\|
		\longrightarrow0,
		\quad\text{as } N \to \infty,
	\end{equation*}
	with probability one. 
	Since $m$ is finite,
	\begin{equation*}
		\sup_{\mathbf z\in\mathcal K}
		\|G^N(\mathbf z)-G(\mathbf z)\|_F
		\longrightarrow0,
		\quad\text{as } N \to \infty,
	\end{equation*}
	with probability one. 
	Lemma~\ref{lem:D:thetaX-lip-H} then yields
	\begin{equation*}
		\sup_{\mathbf z\in\mathcal K}
		\bigl|\Theta_{\mathcal X}^N(\mathbf z)-\Theta_{\mathcal X}(\mathbf z)\bigr|
		\le
		\sup_{\mathbf z\in\mathcal K}
		\|G^N(\mathbf z)-G(\mathbf z)\|_F,
	\end{equation*}
	which proves the result.
\end{proof}

Previous SAA analysis for stochastic multi-objective optimization primarily focuses on optimal or approximately optimal solutions \cite{FliegeXu2011}. 
Here we study consistency directly in terms of the Pareto-stationarity residual used by our algorithm.

\begin{theorem}
	\label{thm:D:population-stationarity}
	Let $\{\varepsilon_N\}$ be a sequence with $\varepsilon_N \downarrow 0$ as $N \to \infty$. 
	For each sample size $N$, we denote by $\widehat{\mathbf z}_N$ the output of Algorithm~\ref{alg:projected-SAA-RMGDA} such that
	\begin{equation}
		\label{eq:D:vanishing-SAA-residual}
		\Theta_{\mathcal X}^N(\widehat{\mathbf z}_N)\le\varepsilon_N.
	\end{equation}
	Then, with probability one, every accumulation point $\mathbf z^\star$ of $\{\widehat{\mathbf z}_N\}$ is Pareto stationary for problem~\eqref{pb:multi}. 
\end{theorem}

\begin{proof}
	Theorem~\ref{thm:C:constrained-residual-complexity} guarantees that Algorithm~\ref{alg:projected-SAA-RMGDA} is able to find an iterate $\widehat{\mathbf z}_N$ such that \eqref{eq:D:vanishing-SAA-residual} holds. 
	Let $\{\widehat{\mathbf z}_{N_j}\}$ be a subsequence of $\widehat{\mathbf z}_N$ converging to $\mathbf z^\star$. 
	Since the convergent subsequence is bounded, there exists a positive integer $r$ such that $\|\widehat{\mathbf z}_{N_j}\| \leq r$ for all $j$ and $\|\mathbf z^\star\| \leq r$. 
	Then we apply Theorem~\ref{thm:D:thetaX-uniform} to the compact set
	\begin{equation*}
		\mathcal K_r := \mathcal X \cap \{\mathbf z \in \mathbb R^n : \|\mathbf z\| \leq r\}.
	\end{equation*}
	Taking the countable intersection of the corresponding probability-one events, we may work on an event on which the asserted uniform convergence holds simultaneously for every nonempty $\mathcal K_r$. 
	Consequently, we have
	\begin{equation*}
		\Theta_{\mathcal X}(\widehat{\mathbf z}_{N_j})
		\le
		\Theta_{\mathcal X}^{N_j}(\widehat{\mathbf z}_{N_j})
		+
		\sup_{\mathbf z\in\mathcal K_r}
		\bigl|\Theta_{\mathcal X}^{N_j}(\mathbf z)-\Theta_{\mathcal X}(\mathbf z)\bigr|
		\longrightarrow0,
	\end{equation*}
	as $j \to \infty$. 
	Now we choose $\lambda_j\in\Delta_m$ and $v_j\in\mathcal N_{\mathcal X}(\widehat{\mathbf z}_{N_j})$ such that
	\begin{equation*}
		\|G(\widehat{\mathbf z}_{N_j})\lambda_j+v_j\|
		\le
		\Theta_{\mathcal X}(\widehat{\mathbf z}_{N_j})+j^{-1}.
	\end{equation*}
	Since $G$ is bounded on $\mathcal K_r$, the sequence $\{v_j\}$ is bounded.
	After passing to a further subsequence, we assume that $\lambda_j\to\lambda^\star\in\Delta_m$ and $v_j\to v^\star$ as $j \to \infty$. 
	Since the graph of the normal-cone mapping to the closed convex set $\mathcal X$ is closed, we have
	\begin{equation*}
		v^\star\in\mathcal N_{\mathcal X}(\mathbf z^\star).
	\end{equation*}
	Passing to the limit gives
	\begin{equation*}
		G(\mathbf z^\star)\lambda^\star+v^\star=0.
	\end{equation*}
	Therefore, $\mathbf z^\star$ is Pareto stationary for problem~\eqref{pb:multi}. 
\end{proof}

\begin{remark}
	Theorem~\ref{thm:D:population-stationarity} characterizes accumulation points but does not guarantee their existence.
	The existence of an accumulation point follows from any additional condition that makes the output sequence relatively compact.
	For example, this is immediate when $\mathcal X$ is compact.
	More generally, with the scalar function $\Phi^N$ defined in Section~\ref{sec:C:complexity}, suppose that, with probability one, there exist $N_0\in\mathbb N$, a constant $B_0<\infty$, and a lower semicontinuous level-bounded function $\underline\Phi:\mathcal X\to\mathbb R$ such that
	\begin{equation*}
		\underline\Phi(\mathbf z)\le\Phi^N(\mathbf z)\quad\forall\mathbf z\in\mathcal X,
		\qquad
		\Phi^N(\mathbf z_0)\le B_0,
		\qquad N\ge N_0.
	\end{equation*}
	Then Lemma~\ref{lem:C:constrained-descent} and $S_\infty = \sum_{k=0}^{\infty}\rho_k<\infty$ imply that every algorithmic output belongs to the compact set
	\begin{equation*}
		\left\{
		\mathbf z\in\mathcal X:
		\underline\Phi(\mathbf z)
		\le
		B_0+\frac{\kappa_m}{\sigma_0}S_\infty
		\right\}.
	\end{equation*}
	Thus $\{\widehat{\mathbf z}_N\}$ has at least one accumulation point, and Theorem~\ref{thm:D:population-stationarity} applies to each of them.
\end{remark}

\subsection{Finite-sample quantification on a compact set}

\label{sec:D:finiteN}

The preceding consistency result is asymptotic. 
Motivated by finite-sample analysis of SAA in stochastic programming and generalized equations \citep{xu2010uniform,Chen_Shapiro_Sun,Chen_shen}, we derive an explicit high-probability bound for the gradient approximation on a compact set and transfer it to the Pareto-stationarity residual. 
The argument uses a standard covering-number construction together with Hoeffding's inequality. 

\begin{assumption}
	\label{ass:finite-sample-gradient}
	Let $\mathcal K\subseteq\mathcal X$ be a nonempty compact set.
	For each $i\in[m]$, there exist constants $\bar M_{i,\mathcal K}<\infty$ and $L_{i,\mathcal K}<\infty$ such that, for $\mathbb P$-a.e. $\xi\in\Xi$,
	\begin{equation*}
		\sup_{\mathbf z\in\mathcal K}\|\nabla f_i(\mathbf z,\xi)\|
		\le\bar M_{i,\mathcal K},
	\end{equation*}
	and
	\begin{equation*}
		\|\nabla f_i(\mathbf z_1,\xi)-\nabla f_i(\mathbf z_2,\xi)\|
		\le L_{i,\mathcal K}\|\mathbf z_1-\mathbf z_2\|
		\quad\forall\mathbf z_1,\mathbf z_2\in\mathcal K.
	\end{equation*}
	Define
	\begin{equation*}
		\bar M_{\mathcal K}:=\max_{i\in[m]}\bar M_{i,\mathcal K},
		\qquad
		L_{\mathcal K}:=\max_{i\in[m]}L_{i,\mathcal K}.
	\end{equation*}
\end{assumption}

Assumption~\ref{ass:finite-sample-gradient} strengthens the integrable-envelope condition in Assumption~\ref{ass:standing} only for the finite-sample estimate. 
For $\eta>0$, let $\mathcal N(\mathcal K,\eta)$ denote the covering number of $\mathcal K$.

\begin{theorem}
	\label{thm:D:finite-uniform-G}
	Fix $\delta\in(0,1)$ and $\varepsilon_G>0$.
	Suppose Assumption~\ref{ass:finite-sample-gradient} holds and $L_{\mathcal K}>0$.
	Define
	\begin{equation*}
		\eta:=\frac{\varepsilon_G}{4L_{\mathcal K}\sqrt m}.
	\end{equation*}
	If
	\begin{equation}
		\label{eq:D:finiteN-G}
		N\ge
		\frac{32nm\,\bar M_{\mathcal K}^2}{\varepsilon_G^2}
		\left[
		\log\left(\frac{2nm}{\delta}\right)
		+\log\mathcal N(\mathcal K,\eta)
		\right],
	\end{equation}
	then
	\begin{equation*}
		\mathbb P\left(
		\sup_{\mathbf z\in\mathcal K}\|G^N(\mathbf z)-G(\mathbf z)\|_F
		\le\varepsilon_G
		\right)\ge1-\delta.
	\end{equation*}
\end{theorem}

\begin{proof}
	Let $\{\mathbf z^1,\ldots,\mathbf z^Q\}$ be an $\eta$-net of $\mathcal K$, where $Q=\mathcal N(\mathcal K,\eta)$.
	For every objective $i$, net point $r$, and coordinate $\ell$, the centered sample-gradient coordinate is bounded in absolute value by $2\bar M_{\mathcal K}$.
	Hoeffding's inequality \citep[Section~2.2]{high_dim_Prob}, followed by a union bound over $n m Q$ choices, shows that condition~\eqref{eq:D:finiteN-G} implies, with probability at least $1-\delta$,
	\begin{equation*}
		\max_{i\in[m]}\max_{1\le r\le Q}
		\|\nabla F_i^N(\mathbf z^r)-\nabla F_i(\mathbf z^r)\|
		\le\frac{\varepsilon_G}{2\sqrt m}.
	\end{equation*}
	For any $\mathbf z\in\mathcal K$, choose $\mathbf z^r$ with $\|\mathbf z-\mathbf z^r\|\le\eta$.
	The two Lipschitz bounds in Assumption~\ref{ass:finite-sample-gradient} then give
	\begin{equation*}
		\|\nabla F_i^N(\mathbf z)-\nabla F_i(\mathbf z)\|
		\le
		\frac{\varepsilon_G}{2\sqrt m}+2L_{\mathcal K}\eta
		=
		\frac{\varepsilon_G}{\sqrt m}.
	\end{equation*}
	Summing the squares over $i \in [m]$ proves the result.
\end{proof}

\begin{remark}
	If $L_{\mathcal K}=0$, the gradient mappings are constant on $\mathcal K$, and the same conclusion follows directly from the pointwise Hoeffding bound without a covering argument.
\end{remark}

\begin{corollary}
	\label{cor:D:finite-uniform-thetaX}
	Fix $\delta\in(0,1)$ and $\varepsilon>0$.
	Suppose Assumption~\ref{ass:finite-sample-gradient} holds.
	If $L_{\mathcal K}>0$, suppose that $N$ satisfies \eqref{eq:D:finiteN-G} with $\varepsilon_G=\varepsilon$.
	If $L_{\mathcal K}=0$, suppose instead that
	\begin{equation*}
		N\ge
		\frac{32nm\,\bar M_{\mathcal K}^2}{\varepsilon^2}
		\log\left(\frac{2nm}{\delta}\right).
	\end{equation*}
	Then
	\begin{equation*}
		\mathbb P\left(
		\sup_{\mathbf z\in\mathcal K}
		\bigl|\Theta_{\mathcal X}^N(\mathbf z)-\Theta_{\mathcal X}(\mathbf z)\bigr|
		\le\varepsilon
		\right)\ge1-\delta.
	\end{equation*}
\end{corollary}

\begin{proof}
	When $L_{\mathcal K}>0$, the result follows directly from Lemma~\ref{lem:D:thetaX-lip-H} and Theorem~\ref{thm:D:finite-uniform-G}.
	When $L_{\mathcal K}=0$, use the pointwise bound described in the preceding remark and then apply Lemma~\ref{lem:D:thetaX-lip-H}.
\end{proof}

\section{Numerical Experiments}

\label{sec:experiment}

In this section, we compare SAA-RMGDA with the stochastic multi-gradient method (SMG) of Liu and Vicente~\citep{SMGD-vincente}. 
The experiments cover unconstrained and constrained synthetic problems, binary classification, portfolio selection, multi-task molecular prediction, and robot control. 
They assess stationarity, feasibility, objective trade-offs, out-of-sample performance, and computational time.

\subsection{Implementation details}

Unless stated otherwise, the comparison includes SMG-adaptive with stepsize $\alpha_k=1/(k+1)$, SMG-Armijo with a multi-objective Armijo line-search procedure, and SAA-RMGDA from Algorithm~\ref{alg:projected-SAA-RMGDA}. 
All methods use the same initialization, training sample, and iteration budget within each experiment. 
The suffixes \texttt{exp} and \texttt{poly} identify exponential and polynomial Tikhonov schedules. 
The suffixes \texttt{PG} and \texttt{PD} identify the projected-gradient and primal-dual subproblem solvers in Remark~\ref{rem:compute-projected-cauchy}.

SAA-RMGDA updates $\sigma_k$ as prescribed by Algorithm~\ref{alg:projected-SAA-RMGDA} and uses one of the schedules
\begin{equation}
	\label{eq:rho}
	\rho_k=\frac{\rho_0}{(k+1)^{\beta_\rho}}
	\quad\text{or}\quad
	\rho_k=\rho_0\gamma^k,
\end{equation}
where $\rho_0>0$, $\beta_\rho>0$, and $\gamma\in(0,1)$.
The experiment-specific parameters are reported below.
The convergence theory requires $\sum_{k=0}^{\infty}\rho_k<\infty$, which holds for the exponential schedule and for polynomial schedules with $\beta_\rho>1$.
Cases with $\beta_\rho\le1$ are included only to study behavior outside this guarantee.

For unconstrained problems, we set
\begin{equation*}
	\mathbf d_k:=-G^N(\mathbf z_k)\lambda_k=\sigma_k\mathbf p_k,
\end{equation*}
where $\lambda_k$ is the computed multiplier. 
We use $\|\mathbf d_k\|$ as a stationarity surrogate.
For constrained problems, we report $\Theta_{\mathcal X}^N(\mathbf z_k)$ directly.
Final training objectives, problem-specific test metrics, and wall-clock time supplement these stationarity measures.

The experiments in Subsections~\ref{subsec:synthetic}--\ref{subsec:molecular} are carried out on a MacBook with Apple M1 8-core CPU and 8GB of memory. 
The continuous robot control experiment in Subsection~\ref{subsec:robot} is conducted on a MacBook with Apple M4 10-core CPU, 10-core GPU, and 24GB of memory.

\subsection{Synthetic stochastic problems}

\label{subsec:synthetic}

We begin with the FF1, MOP2, and MOP3 synthetic problems \citep{FF1}.
FF1 is unconstrained, while MOP2 and MOP3 have box constraints.
Stochasticity is introduced through an additive perturbation of the decision vector.
For
\begin{equation*}
	\phi(\mathbf z):=(\phi_1(\mathbf z),\phi_2(\mathbf z))^\top,
	\qquad \mathbf z\in\mathcal X\subseteq\mathbb R^d,
\end{equation*}
we draw $\xi\sim\mathrm{Unif}([-\delta_\xi,\delta_\xi]^d)$ and define
\begin{equation*}
	f_i(\mathbf z,\xi):=\phi_i(\mathbf z+\xi),
	\qquad
	F_i(\mathbf z):=\mathbb E[f_i(\mathbf z,\xi)],
	\qquad i=1,2.
\end{equation*}
Given i.i.d.\ samples $\{\xi_j\}_{j=1}^N$, the SAA objectives are
\begin{equation*}
	F_i^N(\mathbf z)
	=
	\frac1N\sum_{j=1}^N \phi_i(\mathbf z+\xi_j),
	\ i=1,2.
\end{equation*}

For FF1, we set $d=2$ and $\mathcal X=\mathbb R^2$ with
\begin{equation*}
	\begin{aligned}
		\phi_1(\mathbf z)
		&=1-\exp\left(-(z_1-1)^2-(z_2+1)^2\right),\\
		\phi_2(\mathbf z)
		&=1-\exp\left(-(z_1+1)^2-(z_2-1)^2\right).
	\end{aligned}
\end{equation*}
Their minimizers are $(1,-1)^\top$ and $(-1,1)^\top$, respectively, so the smooth nonconvex objectives induce a nontrivial trade-off.

For MOP2, we set $d=2$, $\mathcal X=[-4,4]^2$, and
\begin{equation*}
	c:=\frac1{\sqrt d}(1,\ldots,1)^\top.
\end{equation*}
The deterministic objectives are
\begin{equation*}
	\phi_1(\mathbf z)=1-\exp\bigl(-\|\mathbf z-c\|^2\bigr),
	\qquad
	\phi_2(\mathbf z)=1-\exp\bigl(-\|\mathbf z+c\|^2\bigr).
\end{equation*}

For MOP3, we set $d=2$ and $\mathcal X=[-\pi,\pi]^2$.
Define
\begin{equation*}
	\begin{aligned}
		A_1&=\frac12\sin(1)-2\cos(1)+\sin(2)-\frac32\cos(2),\\
		A_2&=\frac32\sin(1)-\cos(1)+2\sin(2)-\frac12\cos(2),
	\end{aligned}
\end{equation*}
and, for $\mathbf z=(z_1,z_2)^\top$,
\begin{equation*}
	\begin{aligned}
		B_1(\mathbf z)
		&=\frac12\sin(z_1)-2\cos(z_1)+\sin(z_2)-\frac32\cos(z_2),\\
		B_2(\mathbf z)
		&=\frac32\sin(z_1)-\cos(z_1)+2\sin(z_2)-\frac12\cos(z_2).
	\end{aligned}
\end{equation*}
The objectives are
\begin{equation*}
	\begin{aligned}
		\phi_1(\mathbf z)
		&=-1-\bigl(A_1-B_1(\mathbf z)\bigr)^2
		-\bigl(A_2-B_2(\mathbf z)\bigr)^2,\\
		\phi_2(\mathbf z)
		&=-(z_1+3)^2-(z_2+1)^2.
	\end{aligned}
\end{equation*}
Both MOP2 and MOP3 are smooth nonconvex stochastic biobjective problems with box constraints.

We report $\|\mathbf d_k\|$ for FF1 and the SAA residual $\Theta_{\mathcal X}^N(\mathbf z)$ from~\eqref{eq:A:thetaXN} for the box-constrained problems.
Their feasibility violation is measured by
$
\operatorname{dist}(\mathbf z,\mathcal X)
=
\|\mathbf z-\operatorname{Proj}_{\mathcal X}(\mathbf z)\|
=
\left\|
\max\{\mathbf l-\mathbf z,0\}+\max\{\mathbf z-\mathbf u,0\}
\right\|$ 
with $\mathcal X=[\mathbf l,\mathbf u]$,
where the maxima are taken componentwise.

All three problems use $N=10000$, $\delta_\xi=0.1$, the common initialization $\mathbf z_0=0$, and a common iteration budget.
For FF1, we test $\rho_0\in\{1,10^{-3}\}$ with $\beta_\rho=0.75$ and $\gamma=0.99$.
For MOP2 and MOP3, we use $\rho_0=0.1$ and $\gamma=0.99$.
The polynomial exponents are $0.51$ for MOP2 and $1.20$ for MOP3.

\begin{figure}[ht]
	\centering
	\begin{subfigure}{0.32\linewidth}
		\centering
		\includegraphics[width=\linewidth]{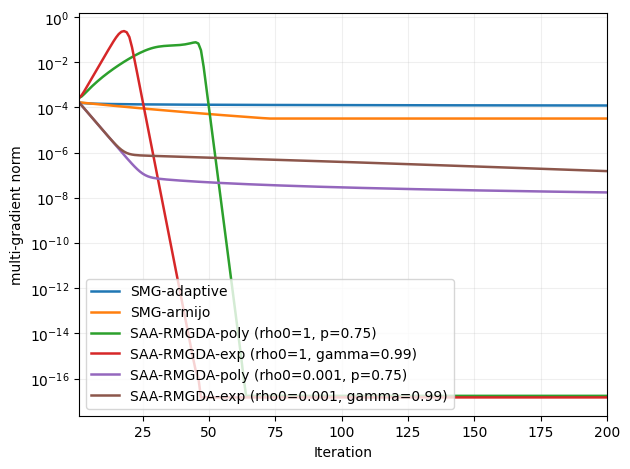}
		\caption{FF1.}
		\label{fig:ff1_norm}
	\end{subfigure}
	\hfill
	\begin{subfigure}{0.32\linewidth}
		\centering
		\includegraphics[width=\linewidth]{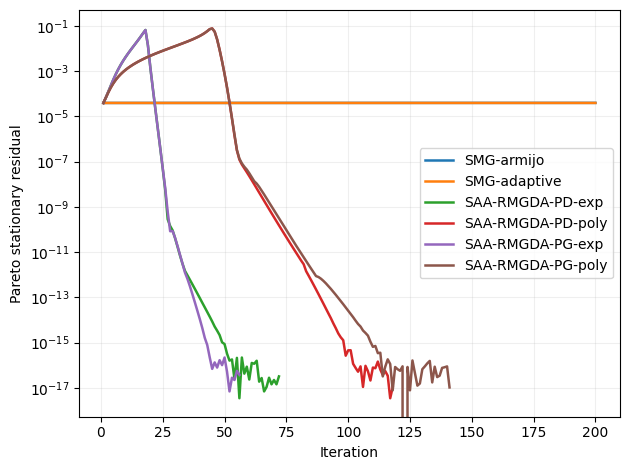}
		\caption{MOP2.}
		\label{fig:mop2_residual}
	\end{subfigure}
	\hfill
	\begin{subfigure}{0.32\linewidth}
		\centering
		\includegraphics[width=\linewidth]{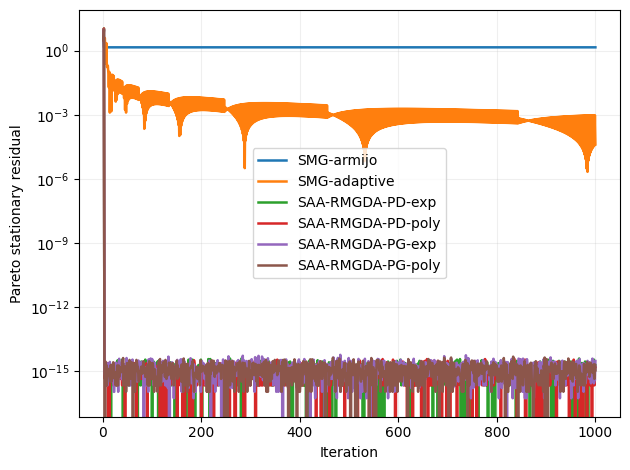}
		\caption{MOP3.}
		\label{fig:mop3_residual}
	\end{subfigure}
	\caption{
		Stationarity surrogate or residual versus iteration number for SMG-adaptive, SMG-Armijo, and SAA-RMGDA on the synthetic stochastic test problems.
		For FF1 we plot the multi-gradient norm $\|\mathbf d_k\|$; for MOP2 and MOP3 we plot the box-constrained Pareto-stationarity residual.}
	\label{fig:synthetic_results}
\end{figure}

\begin{table}[ht]
	\centering
	\begin{tabular}{llccc}
		\hline
		\textbf{problem}
		& \textbf{method}
		& \textbf{time (sec)}
		& \textbf{final $F_1^N$}
		& \textbf{final $F_2^N$} \\
		\hline
		FF1
		& SAA-RMGDA-exp $(\rho_0=10^{-3})$  & 0.1414 & 0.8612  & 0.8609  \\
		& SAA-RMGDA-poly $(\rho_0=10^{-3})$ & 0.1417 & 0.8612  & 0.8610  \\
		& SAA-RMGDA-exp $(\rho_0=1)$        & \textbf{0.1276} & 0.9996  & 0.0261   \\
		& SAA-RMGDA-poly $(\rho_0=1)$       & 0.1530 & 0.9996  & 0.0261   \\
		& SMG-Armijo                          & 0.5523 & 0.8612  & 0.8610  \\
		& SMG-adaptive                        & 0.1793 & 0.8612  & 0.8610  \\
		\hline
		MOP2
		& SAA-RMGDA-PD-exp  & 0.1624 & 0.0085  & 0.9778   \\
		& SAA-RMGDA-PD-poly & 0.4245 & 0.0085  & 0.9778   \\
		& SAA-RMGDA-PG-exp  & \textbf{0.1179} & 0.0085  & 0.9778   \\
		& SAA-RMGDA-PG-poly & 0.5205 & 0.0085  & 0.9778  \\
		& SMG-Armijo        & 0.7627 & 0.9299  & 1.0000   \\
		& SMG-adaptive      & 0.6056 & 0.9301  & 1.0000   \\
		\hline
		MOP3
		& SAA-RMGDA-PD-exp  & 2.2795 & -7.9998  & -38.4056  \\
		& SAA-RMGDA-PD-poly & 1.9805 & -7.4053  & -38.5033  \\
		& SAA-RMGDA-PG-exp  & 2.8344 & -13.6526 & -37.7787 \\
		& SAA-RMGDA-PG-poly & \textbf{1.9413} & -12.7427 & -37.8460  \\
		& SMG-Armijo        & 18.9592 & -43.1496 & -11.3994  \\
		& SMG-adaptive      & 10.6419 & -61.1723 & -6.1232  \\
		\hline
	\end{tabular}
	\caption{
		Comparison of wall-clock time and final objective values for SMG and SAA-RMGDA on the synthetic stochastic test problems.
		For the box-constrained MOP2 and MOP3 problems, all methods have zero feasibility violation up to numerical precision.}
	\label{tab:synthetic_results}
\end{table}

Figure~\ref{fig:synthetic_results} shows lower final stationarity quantities for the SAA-RMGDA variants on all three problems.
On MOP2 and MOP3, these variants drive the constrained residual close to machine precision while retaining feasibility.

Table~\ref{tab:synthetic_results} also shows that the methods can approach different trade-off points.
On FF1, $\rho_0=10^{-3}$ yields objective values close to those of the SMG baselines, whereas $\rho_0=1$ favors the second objective.
All SAA-RMGDA variants reach the same reported point on MOP2 and require less time than the SMG methods.
On MOP3, the methods terminate in distinct trade-off regions, so the timing comparisons should be interpreted together with the final objective values.

\subsection{Binary classification problems}

\label{subsec:classification}



This group of experiments follows the biobjective logistic-regression setting of Liu and Vicente \citep{SMGD-vincente}, where the two objectives are the regularized logistic losses on two subgroups.
Each sample is $\xi_j=(a_j,y_j,g_j)$, where $a_j\in\mathbb R^n$ is the feature vector, $y_j\in\{-1,1\}$ is the label, and $g_j\in\{1,2\}$ is the group indicator.
We use the linear classifier $a\mapsto x^\top a+b$, with
\begin{equation*}
	\mathbf z:=(x^\top,b)^\top\in\mathbb R^{n+1},
	\qquad
	\ell(\mathbf z;a,y)
	:=
	\log\left(1+\exp\bigl(-y(x^\top a+b)\bigr)\right).
\end{equation*}
Let $J_i:=\{j:g_j=i\}$, $i=1,2$.
The SAA objectives are
\begin{equation*}
	F_i^N(\mathbf z)
	=
	\frac{1}{|J_i|}
	\sum_{j\in J_i}
	\ell(\mathbf z;a_j,y_j)
	+
	\frac{\mu_i}{2}\|x\|^2,
	\qquad i=1,2.
\end{equation*}
Here $\mu_i\ge0$ is the regularization parameter for subgroup $i$.
The resulting biobjective problem is
\begin{equation*}
	\min_{\mathbf z\in\mathbb R^{n+1}}
	F^N(\mathbf z),
	\qquad
	F^N(\mathbf z)
	:=
	\begin{pmatrix}
		F_1^N(\mathbf z)\\[1mm]
		F_2^N(\mathbf z)
	\end{pmatrix}.
\end{equation*}
Both objectives are smooth and convex, and they are generally conflicting because they measure prediction losses on different subgroups.

We first test this formulation on the Wisconsin Breast Cancer dataset \citep{dua2017breastcancer}.
We use an 80:20 training--test split and initialize all methods at $\mathbf z_0=0$.
For SAA-RMGDA, we use
$
\rho_0=1,\ \gamma=0.99,\ \beta_\rho=0.75.
$

\begin{figure}[ht]
	\centering
	\includegraphics[width=0.42\linewidth]{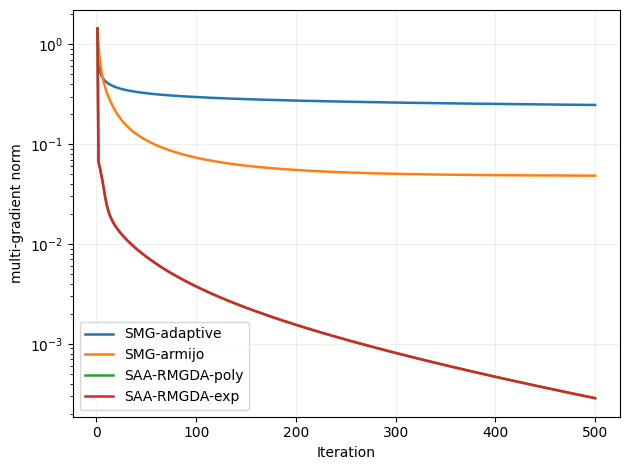}
	\caption{
		Multi-gradient norm $\|\mathbf d_k\|$ versus iteration number for SMG-adaptive, SMG-Armijo, and SAA-RMGDA on the Wisconsin Breast Cancer dataset.}
	\label{fig:bc_norm}
\end{figure}

\begin{table}[ht]
	\centering
	\begin{tabular}{lcccc}
		\hline
		\textbf{method}
		& \textbf{time (sec)}
		& \textbf{train $F_1^N$}
		& \textbf{train $F_2^N$}
		& \textbf{test accuracy} \\
		\hline
		SAA-RMGDA-exp
		& \textbf{1.0481} & \textbf{0.0833} & \textbf{0.1084} & \textbf{0.9739} \\
		SAA-RMGDA-poly
		& 2.1702 & \textbf{0.0833} & \textbf{0.1084} & \textbf{0.9739} \\
		SMG-Armijo
		& 4.4801 & 0.0986 & 0.1305 & 0.9565 \\
		SMG-adaptive
		& 3.9940 & 0.1983 & 0.2315 & 0.9478 \\
		\hline
	\end{tabular}
	\caption{
		Comparison of wall-clock time, training objective values, and test accuracy on the Wisconsin Breast Cancer dataset.}
	\label{tab:bc_table}
\end{table}

Figure~\ref{fig:bc_norm} shows a faster decrease of $\|\mathbf d_k\|$ for the two SAA-RMGDA schedules than for the SMG baselines.
Their final training losses and test accuracies are also better in Table~\ref{tab:bc_table} for this split.

We also test the methods on the four benchmark datasets \texttt{heart}, \texttt{australian}, \texttt{svmguide3}, and \texttt{german.numer} used in \citep{SMGD-vincente}.
For SAA-RMGDA, we use the exponential schedule with $\rho_0=1$ and $\gamma=0.99$.
All methods use the same initialization, training sample size, and iteration budget of $250$ iterations.

\begin{table}[ht]
	\centering
	\begin{tabular}{llcccc}
		\hline
		\textbf{dataset}
		& \textbf{method}
		& \textbf{time (sec)}
		& \textbf{train $F_1^N$}
		& \textbf{train $F_2^N$}
		& \textbf{test accuracy} \\
		\hline
		\texttt{heart}
		& SAA-RMGDA-exp  & \textbf{0.0787} & \textbf{0.3605} & 0.3724 & \textbf{0.8889} \\
		& SMG-Armijo     & 0.6807 & 0.3720 & \textbf{0.3722} & 0.8704 \\
		& SMG-adaptive   & 0.8039 & 0.4724 & 0.4727 & 0.8519 \\
		\hline
		\texttt{australian}
		& SAA-RMGDA-exp  & \textbf{0.0896} & \textbf{0.3513} & \textbf{0.3494} & 0.9065 \\
		& SMG-Armijo     & 2.1593 & 0.3626 & 0.3621 & \textbf{0.9137} \\
		& SMG-adaptive   & 2.2269 & 0.4861 & 0.4856 & 0.8777 \\
		\hline
		\texttt{svmguide3}
		& SAA-RMGDA-exp  & \textbf{0.1213} & \textbf{0.4162} & \textbf{0.4453} & \textbf{0.8120} \\
		& SMG-Armijo     & 1.7005 & 0.4529 & 0.4612 & 0.7880 \\
		& SMG-adaptive   & 0.2974 & 0.5819 & 0.5876 & 0.8000 \\
		\hline
		\texttt{german.numer}
		& SAA-RMGDA-exp  & \textbf{0.0900} & \textbf{0.4751} & \textbf{0.4536} & \textbf{0.7550} \\
		& SMG-Armijo     & 1.3128 & 0.4782 & 0.4780 & 0.7450 \\
		& SMG-adaptive   & 1.3119 & 0.5888 & 0.5886 & 0.7500 \\
		\hline
	\end{tabular}
	\caption{
		Comparison of wall-clock time, training objective values, and test accuracy on four benchmark binary classification datasets.}
	\label{tab:libsvm_logistic}
\end{table}

Table~\ref{tab:libsvm_logistic} reports the smallest running time for SAA-RMGDA on all four datasets.
It also gives the lowest pair of training losses in three cases and competitive test accuracy across the four splits.

\subsection{Simplex-constrained portfolio selection}

\label{subsec:portfolio}

We next consider a stochastic portfolio selection problem with a simplex constraint \citep{markowitz1952portfolio}.
Let $R^j\in\mathbb R^d$, $j=1,\ldots,N$, denote i.i.d. return samples, where $R_i^j$ is the single-period return of asset $i$ in scenario $j$.
The decision variable is the portfolio weight vector $\mathbf z\in\Delta_d$, so that the feasible set represents long-only fully invested portfolios.

For a portfolio $\mathbf z$, the realized return in scenario $j$ is $(R^j)^\top\mathbf z$.
We consider the finite-sample biobjective problem
$\min_{\mathbf z\in\Delta_d} F^N(\mathbf z) := (F_1^N(\mathbf z), F_2^N(\mathbf z))^\top$,
where
\begin{equation*}
	F_1^N(\mathbf z)
	:=
	-\frac1N\sum_{j=1}^N (R^j)^\top\mathbf z,
	\text{ and }
	F_2^N(\mathbf z)
	:=
	\frac1N\sum_{j=1}^N
	\left(
	\operatorname{softplus}_{\tau}\bigl(-(R^j)^\top\mathbf z\bigr)
	\right)^2.
\end{equation*}
Here $F_1^N$ corresponds to maximizing the sample average portfolio return, whereas $F_2^N$ is a smooth downside-risk surrogate.
We use
$
\operatorname{softplus}_{\tau}(t)
:=
\tau\log\left(1+\exp\left(\frac{t}{\tau}\right)\right)
$
with $\tau>0$.
Thus, $\operatorname{softplus}_{\tau}(-(R^j)^\top\mathbf z)$ is a smooth approximation of $\max\{-(R^j)^\top\mathbf z,0\}$.
The objectives express the usual return-risk conflict.
Reducing $F_1^N$ favors a larger average return, whereas reducing $F_2^N$ favors less downside risk.

We generate the return scenarios from a bounded factor model.
For $i=1,\ldots,d$, let
$
s_i:=\frac{i-1}{d-1}.
$
The expected return, idiosyncratic volatility, and factor loading scale of asset $i$ are defined by
$
\mu_i=\mu_{\min}+(\mu_{\max}-\mu_{\min})s_i,
\
\sigma_i=\sigma_{\min}+(\sigma_{\max}-\sigma_{\min})s_i,
$
and
$
\beta_i=\beta_{\min}+(\beta_{\max}-\beta_{\min})s_i.
$
Thus, assets with larger expected returns also have higher volatility.
Let $r$ be the number of common factors.
The factor loading matrix $B\in\mathbb R^{d\times r}$ is defined by
$
B_{i\ell}=\beta_i C_{i\ell},
\
C_{i\ell}\sim
\mathrm{Unif}\left[-1 / {\sqrt r}, 1 / {\sqrt r}\right].
$
For each scenario $j$, we sample
$
u^j\sim \mathrm{Unif}([-1,1]^r),
\
\varepsilon_i^j\sim \mathrm{Unif}([-\sigma_i,\sigma_i]),
\ i=1,\ldots,d,
$
and set
$
R^j=\mu+Bu^j+\varepsilon^j.
$
Since both $u^j$ and $\varepsilon^j$ have bounded support, the generated return vector $R^j$ is bounded.
Therefore, the sample-wise objectives are smooth and bounded on the simplex.

In the experiments, we set $d=20$, $r=3$, and $N=2000$, and use an 80:20 training--test split.
The return-model parameters are
\begin{equation*}
	\mu_{\min}=0.02, \; \mu_{\max}=0.10, \;
	\sigma_{\min}=0.03, \; \sigma_{\max}=0.20, \;
	\beta_{\min}=0.15, \; \beta_{\max}=0.60.
\end{equation*}
For the smooth downside-risk objective, we use $\tau=0.02$.
For SAA-RMGDA, we compare the primal-dual and projected-gradient implementations of the projected regularized multi-gradient step.
For the adaptive regularization schedules, we use $\rho_0=0.1,\ \gamma=0.99,\ \beta_\rho=1.20$.
All methods are initialized from the same feasible point in the simplex and are run for at most $1000$ iterations.
We report the constrained Pareto-stationarity residual, final training objective values, test-set return, test-set downside risk, and wall-clock time.
All methods preserve the simplex constraint up to numerical precision.

\begin{figure}[ht]
	\centering
	\includegraphics[width=0.5\linewidth]{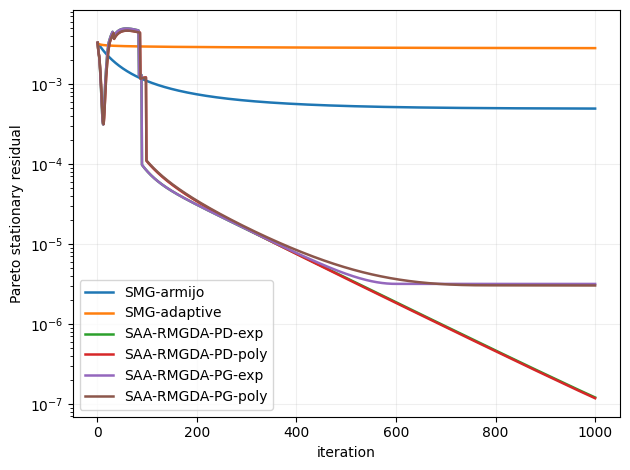}
	\caption{
		Pareto-stationarity residual versus iteration number for SMG-adaptive, SMG-Armijo, and SAA-RMGDA variants on the simplex-constrained portfolio selection problem.}
	\label{fig:portfolio_simplex_residual}
\end{figure}

\begin{table}[ht]
	\centering
	\begin{tabular}{lccccc}
		\hline
		\textbf{method}
		& \textbf{time (sec)}
		& \textbf{final $F_1^N$}
		& \textbf{final $F_2^N$}
		& \textbf{test-set return}
		& \textbf{test-set downside risk} \\
		\hline
		SAA-RMGDA-PD-exp
		& 3.6321 & -0.1044 & 0.0042 & 0.0822 & 0.0058\\
		SAA-RMGDA-PD-poly
		& 15.6006 & -0.1044 & 0.0042 & 0.0822 & 0.0058\\
		SAA-RMGDA-PG-exp
		& \textbf{2.7640} & -0.1044 & 0.0042 & 0.0822 & 0.0058\\
		SAA-RMGDA-PG-poly
		& 5.5887 & -0.1044 & 0.0042 & 0.0822 & 0.0058\\
		SMG-Armijo
		& 4.1733 & -0.0575 & 0.0000 & 0.0585 & 0.0000\\
		SMG-adaptive
		& 3.9854 & -0.0572 & 0.0001 & 0.0611 & 0.0001\\
		\hline
	\end{tabular}
	\caption{
		Comparison of wall-clock time, final objective values, and out-of-sample portfolio performance for SMG and SAA-RMGDA on the simplex-constrained portfolio selection problem.
		All methods have zero feasibility violation up to numerical precision.}
	\label{tab:PM_results}
\end{table}

Figure~\ref{fig:portfolio_simplex_residual} shows lower final residuals for all SAA-RMGDA variants than for the two SMG methods.
The primal-dual curves continue decreasing late in the run, while the SMG curves level off earlier.
Table~\ref{tab:PM_results} shows a common SAA-RMGDA trade-off point with higher test return and higher downside risk than the SMG points.
Among the four variants, the projected-gradient implementation with the exponential schedule has the smallest reported running time.

To further examine the effect of the SAA sample size, we repeat the portfolio experiment with different training sample sizes $N$ and record the held-out expected portfolio return and smooth downside risk in Figure~\ref{fig:portfolio_saa_boxplots}.

\begin{figure}[ht]
	\centering
	\begin{subfigure}{0.4\linewidth}
		\centering
		\includegraphics[width=\linewidth]{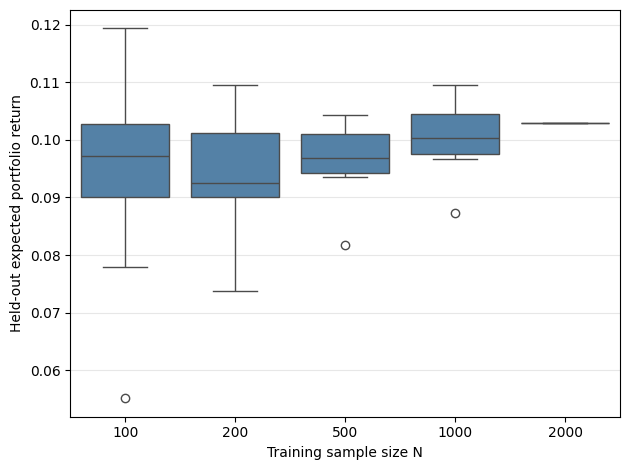}
		\caption{\scriptsize
			Held-out expected portfolio return.}
		\label{fig:portfolio_saa_return}
	\end{subfigure}
	\begin{subfigure}{0.4\linewidth}
		\centering
		\includegraphics[width=\linewidth]{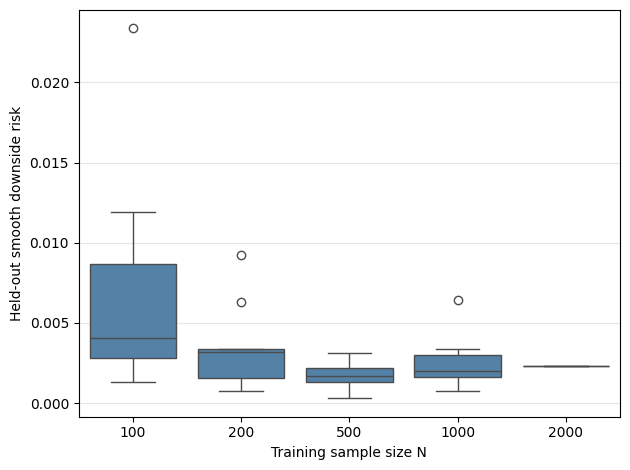}
		\caption{\scriptsize
			Held-out smooth downside risk.}
		\label{fig:portfolio_saa_risk}
	\end{subfigure}
	\caption{
		Held-out performance of SAA-RMGDA on the simplex-constrained portfolio problem for different training sample sizes $N$.}
	\label{fig:portfolio_saa_boxplots}
\end{figure}

As $N$ increases, the held-out return becomes more concentrated.
The downside-risk distribution also shifts downward and narrows.
Together with the zero reported feasibility violation, these results show that the projected method maintains the simplex constraint while exploring return-risk trade-offs.

%

\subsection{Multi-task molecular property prediction}

\label{subsec:molecular}

We next consider a multi-task regression problem on the QM9 molecular property dataset~\citep{ramakrishnan2014quantum}.
QM9 contains small organic molecules with quantum-chemical property labels.
Each molecule is represented as a graph, where nodes encode atom features and edges encode bond information.
We use an NNConv-based graph neural network architecture, following the PyTorch Geometric QM9 example and the multi-task learning setting in \citep{navon2022multitasklearningbargaininggame}.

We formulate the problem as an $11$-objective finite-sum multi-task learning problem. 
The selected QM9 target indices are
$
\mathcal I=\{0,1,2,3,5,6,7,8,9,10,11\}.
$
Since these targets have different physical units and scales, each target is normalized using the mean and standard deviation computed from the full training split.
The same normalization is then applied to the validation and test splits.

Let $\mathcal G_\ell$ denote the graph representation of molecule $\ell$, and let $\widetilde y_\ell\in\mathbb R^{11}$ be the corresponding normalized target vector.
The graph neural network with parameter vector $\theta$ is denoted by
$
h_\theta(\mathcal G_\ell)
=
\bigl(
h_{\theta,1}(\mathcal G_\ell),\ldots,
h_{\theta,11}(\mathcal G_\ell)
\bigr)^\top
\in\mathbb R^{11}.
$
We define the SAA objective vector
\begin{equation*}
	F^N(\theta)
	=
	\bigl(F_1^N(\theta),\ldots,F_{11}^N(\theta)\bigr)^\top,
\end{equation*}
where
$
F_i^N(\theta)
=
\frac1N
\sum_{\ell\in S_N}
\frac12
\left(
h_{\theta,i}(\mathcal G_\ell)-\widetilde y_{\ell,i}
\right)^2,
\ i=1,\ldots,11.
$
Thus, each objective measures the normalized squared prediction error for one molecular property.

We use Softplus activations so that the objective function is smooth.
With fixed graph construction and edge features, affine target normalization, and squared losses, every component $F_i^N$ is differentiable with respect to the trainable parameters.

The dataset is deterministically split into $110{,}000$ training molecules, $10{,}000$ validation molecules, and $10{,}000$ test molecules. In the main QM9 experiment, we set $N=512$, and test on the out-of-sample $10000$ molecules.

We compare the polynomial and exponential SAA-RMGDA variants with two SMG-type baselines that use adaptive and Armijo stepsizes.
For SAA-RMGDA, the regularization parameter is chosen as
$
\rho_0=10^{-2},\ \beta_\rho=0.5,\ \gamma=0.999.
$
All methods are initialized from the same network parameters for each random seed and use the same fixed SAA subset.
We report the Pareto-stationarity measure versus the iteration number, the test normalized mean squared error, and the per-task test mean absolute errors in the original physical units.

\begin{figure}[ht]
	\centering
	\includegraphics[width=0.5\linewidth]{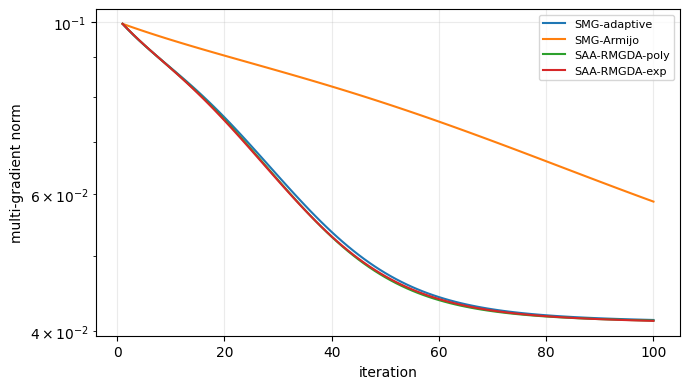}
	\caption{
		Multi-gradient norm versus iteration number for SMG-adaptive, SMG-Armijo, and SAA-RMGDA on the QM9 multi-task molecular property prediction problem.}
	\label{fig:qm9_chi}
\end{figure}

\begin{table}[ht]
	\centering
	\begin{tabular}{lccccc}
		\hline
		\textbf{method}
		& \textbf{time (hr)} 
		& \textbf{train mean}
		& \textbf{train max}
		& \textbf{train N-MSE}
		& \textbf{test N-MSE}
		\\
		\hline
		SAA-RMGDA-exp
		& \textbf{1.92} & \textbf{0.4971} & \textbf{0.6149}
		& \textbf{0.9941} & \textbf{1.0384} \\
		SAA-RMGDA-poly
		& 2.15 & 0.4975 & 0.6156
		& 0.9951 & 1.0390 \\
		SMG-Armijo
		& 2.20 & 0.4981 & 0.6162
		& 0.9962 & 1.0396  \\
		SMG-adaptive
		& 2.07 & 0.4977 & 0.6158
		& 0.9953 & 1.0391  \\
		\hline
	\end{tabular}
	\caption{
		Summary of final optimization and test metrics on the QM9 multi-task molecular property prediction problem.
		Here ``train mean'' and ``train max'' denote the mean and maximum component values of the final SAA objective vector, and N-MSE denotes normalized mean squared error.}
	\label{tab:qm9_summary}
\end{table}

\begin{table}[ht]
	\centering
	\begin{tabular}{lcccc}
		\hline
		\textbf{task}
		& \textbf{SAA-RMGDA-exp}
		& \textbf{SAA-RMGDA-poly}
		& \textbf{SMG-Armijo}
		& \textbf{SMG-adaptive}
		\\
		\hline
		0 & \textbf{7.2374} & 7.2418 & 7.2457 & 7.2430  \\
		1 & \textbf{3.1998} & 3.2010 & 3.2014 & 3.2013  \\
		2 & 834.9663 & \textbf{834.9524} & 835.0157 & 834.9728  \\
		3 & \textbf{839.3074} & 839.7780 & 840.1189 & 839.9038  \\
		5 & \textbf{0.4456} & 0.4459 & 0.4462 & 0.4460  \\
		6 & 1.0741 & \textbf{1.0741} & 1.0747 & 1.0741  \\
		7 & 1.1793 & 1.1793 &\textbf{1.1762}& 1.1793   \\
		8 & 198.6841 & 198.6438 & \textbf{198.6148} & 198.6318  \\
		9 & \textbf{842.8548} & 843.1283 & 843.4073 & 843.2068  \\
		10 & \textbf{870.0687} & 870.5227 & 870.9440 & 870.5985  \\
		11 & \textbf{0.7361} & 0.7366 & 0.7371 & 0.7368  \\
		\hline
	\end{tabular}
	\caption{
		Per-task test mean absolute errors on QM9 in the original physical units.}
	\label{tab:qm9_task_mae}
\end{table}

Figure~\ref{fig:qm9_chi} shows a steady decrease of the multi-gradient norm for both SAA-RMGDA schedules.
Their final values are close to those of SMG-adaptive, while SMG-Armijo decreases more slowly over the displayed iterations.

Table~\ref{tab:qm9_summary} gives small but consistent advantages to SAA-RMGDA-exp in the mean and maximum SAA objectives, normalized training and test errors, and running time.
The per-task results in Table~\ref{tab:qm9_task_mae} are close across methods.
SAA-RMGDA-exp is best on seven of the eleven listed tasks, while the remaining tasks favor one of the other variants by modest margins.

\subsection{Continuous robot control}

\label{subsec:robot}

We consider a continuous-control problem motivated by the speed-energy trade-off in robotic locomotion~\citep{xu2020prediction}. 
The purpose of this experiment is to evaluate the proposed stochastic multi-objective optimization methods on a robot sequential-control problem in which the decision variable is a neural-network policy and objectives depend on the entire simulated trajectory. 
We use the \texttt{HalfCheetah} model in MuJoCo~\citep{todorov2012mujoco}, simulated using MuJoCo MJX~\citep{mujoco_software}.

Let $s_t\in\mathbb R^{d_s}$ and $a_t\in\mathbb R^{d_a}$ denote the state and action at time $t$, respectively. 
The control policy is parameterized by
\(\theta\in\mathbb R^p\) and is defined by
\[
a_t
=
\pi_\theta(s_t,\varepsilon_t)
:=
\tanh\!\left(
\mu_\theta(s_t)+\sigma\varepsilon_t
\right),
\]
where
\(\mu_\theta:\mathbb R^{d_s}\to\mathbb R^{d_a}\)
is a feedforward neural network with smooth activation functions,
\(\varepsilon_t\sim\mathcal N(0,I_{d_a})\), and \(\sigma>0\) is fixed.
In our implementation, we have \(d_s=17\), \(d_a=6\), and \(\sigma=0.03\). 
Moreover, \(\mu_\theta\) is the network with two hidden layers of width \(16\), where \(\theta\) with dimension \(p=662\) contains all trainable weights and biases of the policy network.

The simulator dynamics are defined by
\[
s_{t+1}=f_{\mathrm{MJX}}(s_t,a_t),
\]
where \(f_{\mathrm{MJX}}\) denotes one control step of the fixed MuJoCo rigid-body system, including joint and actuator dynamics, contact dynamics, and numerical integration.
For each simulation
scenario \(j\), let
\[
\xi_j
=
\left(
s_0^j,
\varepsilon_0^j,\ldots,\varepsilon_{T-1}^j
\right)
\in\mathbb R^{d_s+Td_a}, \text{ where } T=100.
\]
Given \((\theta,\xi_j)\), the policy and simulator generate the trajectory
\[
\tau_\theta(\xi_j)
=
(s_0^j,a_0^j,s_1^j,a_1^j,\ldots,s_T^j).
\]

We then consider the SAA biobjective problem
\[
\min_{\theta\in\mathbb R^p}
F^N(\theta)
:=
\begin{pmatrix}
	F_1^N(\theta)\\[1mm]
	F_2^N(\theta)
\end{pmatrix}.
\]
Let \(x_t(\theta;\xi_j)=[s_t(\theta;\xi_j)]_r\in\mathbb R\) denote the forward position of the
robot at time \(t\) in scenario \(\xi_j\). The corresponding forward
velocity is
\[
v_t(\theta;\xi_j)
:=
\frac{x_{t+1}(\theta;\xi_j)-x_t(\theta;\xi_j)}{\Delta t},
\]
where \(\Delta t=0.05>0\) is the simulation time step. We define \[
F_1^N(\theta)
:=
-\frac1N
\sum_{j=1}^N
\sum_{t=0}^{T-1}
\omega^t
v_t(\theta;\xi_j),
\]
and
\[
F_2^N(\theta)
:=
\frac1N
\sum_{j=1}^N
\sum_{t=0}^{T-1}
\omega^t
\|a_t(\theta;\xi_j)\|^2,
\]
where \(\omega\in(0,1)\) is the discount factor. 
In the experiments, we choose \(\omega=0.99\). 
Thus, minimizing \(F_1^N\) favors larger forward velocity, whereas
minimizing \(F_2^N\) favors lower control energy. For SAA-RMGDA, the regularization parameter is chosen as
$
\rho_0=0.03,\ \beta_\rho=0.5.
$
All methods use the same initialization and the same fixed SAA training sample with
\(
N=64\)
within each replication.
We perform five independent replications. For each replication, an independent set of \(512\) held-out scenarios is used for out-of-sample evaluation and is shared by all competing methods.

We report the final training objectives \(F_1^N\) and \(F_2^N\), wall-clock time, test speed, and test control effort as mean \(\pm\) standard deviation over the five replications. The evolution of the multi-gradient norm \(\|\mathbf d_k\|
\) is illustrated using one representative replication.

\begin{figure}[ht]
	\centering
	\includegraphics[width=0.5\linewidth]{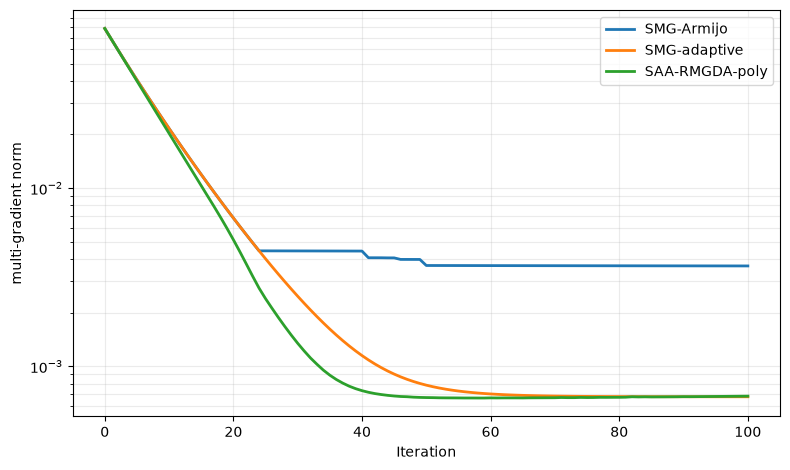}
	\caption{
		Multi-gradient norm versus iteration number for  SMG-Armijo, SMG-adaptive, and SAA-RMGDA-poly for one representative replication of
		the HalfCheetah continuous-control experiment.}
	\label{fig:morl_chi}
\end{figure}

\begin{table}[ht]
	\centering
	\begin{tabular}{lccc}
		\hline
		\textbf{method} & SAA-RMGDA-poly & SMG-Armijo & SMG-adaptive \\
		\hline
		\textbf{time (sec)} 
		& $\mathbf{7742.0 \pm 1.1\times10^3}$ 
		& $13591.9 \pm 3.2\times10^3$ 
		& $8472.0 \pm 1.9\times10^3$ \\
		
		\textbf{train $F_1^N$} 
		& $\mathbf{-3.1065 \pm 0.3385}$ 
		& $-3.1057 \pm 0.3383$ 
		& $-3.1061 \pm 0.3385$ \\
		
		\textbf{train $F_2^N$} 
		& $0.3426 \pm 0.0024$ 
		& $0.3426 \pm 0.0024$ 
		& $0.3426 \pm 0.0024$ \\
		
		\textbf{test forward velocity}
		& $\mathbf{3.0339 \pm 0.0061}$ 
		& $3.0336 \pm 0.0057$ 
		& $3.0338 \pm 0.0062$ \\
		
		\textbf{test energy}
		& $0.3416 \pm 0.0007$ 
		& $0.3416 \pm 0.0007$ 
		& $0.3416 \pm 0.0007$ \\
		\hline
	\end{tabular}
	\caption{
		Mean and standard deviation over five independent replications for the
		HalfCheetah continuous-control experiment.}
	\label{tab:morl_results}
\end{table}

Figure~\ref{fig:morl_chi} illustrates the convergence behavior for one
representative replication. SAA-RMGDA-poly and SMG-adaptive reduce the
multi-gradient norm to a lower level than SMG-Armijo, with
SAA-RMGDA-poly reaching this level slightly faster.

Table~\ref{tab:morl_results} reports the mean and standard deviation over
five independent replications. The three methods attain very similar
training and out-of-sample performance. SAA-RMGDA-poly gives the
largest mean test forward velocity and the smallest mean wall-clock
time, while the test energy values are nearly identical across the
three methods.

\section{Conclusion}

\label{sec:conclusion}

We develop a function-value-free adaptive projected-gradient method for sample-average approximations of constrained stochastic multi-objective problems. 
The regularized subproblem produces feasible update steps and a unique simplex multiplier. 
An adaptive scheme for updating regularization parameters removes the need for line-search procedures. 
Our theoretical analysis gives an $\mathcal O(\varepsilon^{-2})$ iteration complexity for the SAA Pareto-stationarity residual. 
We also prove the consistency of SAA with population Pareto stationarity and a finite-sample residual bound on compact sets. 
The experiments illustrate these properties across several problem classes and also show that the regularization schedule can influence the trade-off point reached by the method. 
Future work may combine our algorithm with increasing or adaptive sample sizes, inexact subproblem stopping rules, and mechanisms for generating a broader approximation of the set of Pareto stationary points or a subset of weakly Pareto optimal solutions. 
Extending the proposed framework to a distributionally robust setting for handling distribution shifts \cite{shapiro2021lectures,wang2025distributionally} would also be an interesting direction.

\bibliographystyle{apalike-arxiv-numeric}
\bibliography{references}
\end{document}